\documentclass[12pt]{article}
\usepackage[top=1in,bottom=1in,right=.8in,left=.8in]{geometry}

\usepackage{amssymb,amsmath,amsthm}
\usepackage{tikz}
\usepackage{arydshln}
\usepackage{verbatim}
\usepackage{titlesec}
\usepackage{bbm}
\usepackage{bm}
\usepackage{hyperref}
\usepackage{blkarray}
\usepackage{subcaption} 
\usepackage{ stmaryrd }
\usepackage{cleveref}
\usepackage{xcolor}
\usepackage{color}
\usepackage[
backend=biber,
style=numeric,
sorting=nyt,
sortcites=true
]{biblatex}
\newcounter{subsubsubsection}[subsubsection]
\renewcommand\thesubsubsubsection{\thesubsubsection.\arabic{subsubsubsection}}

\titleformat{\subsubsubsection}
{\normalfont\normalsize\bfseries}{\thesubsubsubsection}{1em}{}
\titlespacing*{\subsubsubsection}
{0pt}{3.25ex plus 1ex minus .2ex}{1.5ex plus .2ex}

\makeatletter  
\def\toclevel@subsubsubsection{4}
\def\l@subsubsubsection{\@dottedtocline{4}{7em}{4em}}

\makeatother
\newtheorem{thm}{Theorem}[section]
\newtheorem{conjecture}{Conjecture}

\newtheorem{example}{Example}[section]
\newtheorem{remark}{Remark}[section]

\newtheorem{lem}{Lemma}[section]

\newtheorem{proposition}[thm]{Proposition}
\theoremstyle{definition}
\newtheorem{defn}[thm]{Definition}

\newtheorem{corollary}[thm]{Corollary}
\newtheorem{obsv}[thm]{Observation}
\newtheorem{question}[thm]{Question}
\theoremstyle{definition}

\renewenvironment{proof}{\noindent {\bf Proof.}}{\qed}
\allowdisplaybreaks

\title{Rigidity of Circle Packings with Flexible Radii}

\author{Robert Connelly\thanks{Department of Mathematics, Cornell University, Ithaca, USA.} \and  Zhen Zhang\thanks{Center for Applied Mathematics, Cornell University, Ithaca, USA.}}
\date{January 2025}

\begin{document}
	
	\maketitle
	\begin{abstract}
	Circle packings are arrangements of circles that satisfy specified tangency requirements. Many problems about the packing of circles and spheres occur in nature, particularly in material design and protein structure. Rigidity theory studies the realizations of graphs under edge length constraints. This paper tries to bridge classical circle packing problems with modern rigidity theory. We study the rigidity of circle packings representing a given planar graph under radii constraints. We provide sufficient conditions for the packing to be rigid in the first and second order. This gives us a sufficient condition to show a packing is rigid locally. These tools are particularly powerful in optimization problems. Then we will explore the difficulties of global rigidity and explain how the solutions to global rigidity can make progress on some longstanding open problems. 
    
	\end{abstract}
	
	\section{Introduction}
	
	Although the densest packing of equal disks has been solved in dimensions $1,2,3,8$ and $24$\cite{Cohn2016TheSP,Viazovska2016TheSP,Hales2005APO,FejesToth1942}, the packings of unequal disks remain mysterious. Kennedy classified all 9 triangulated packings with 2 sizes \cite{Kennedy2004CompactPO}. Fernique used a non-trivial similar approach to find all triangulated packings with 3 sizes \cite{Fernique2020}. In his 1964 book \textit{Regular Figures} \cite{fejestoth1964}, László Fejes Tóth gave his best guesses of the densest circle packings at given radii ratio: 
\begin{equation*}
    q=r_{min}/r_{max}
\end{equation*}

Gerd Blind showed that as long as $q> \sqrt{\frac{7\tan\tfrac{\pi}{7}-6\tan\tfrac{\pi}{6}}{6\tan\tfrac{\pi}{6}-5\tan\tfrac{\pi}{5}}}\approx 0.74299$, it is not possible to have density higher than the density of hexagonal packing which is $\frac{\pi}{\sqrt{12}}$\cite{blind_1969,blind_1975}. The hexagonal packing is given in Figure \ref{hexagonal}. Since any triangulated packing necessarily beats hexagonal packing in terms of density, there can be no triangulated packing when $q>0.743$ (See \cite{FERNIQUE2025102134} for a review on this topic). If we restrict ourselves to dense triangulated packing only, then we can find a better bound easily: a small disk surrounded by 5 equal disks gives \begin{equation*}
    q = \frac{2\sqrt{2}}{\sqrt{5-\sqrt{5}}}-1\approx 0.70130
\end{equation*}
The triangulated packing with the largest known $q$ given by László Fejes Tóth only has $q\approx 0.6375$\cite{fejestoth1964}. A finite part of this packing is shown in Figure \ref{fejestothhexagon}. 

\begin{figure}
		\centering
		\includegraphics[height=7.5cm]{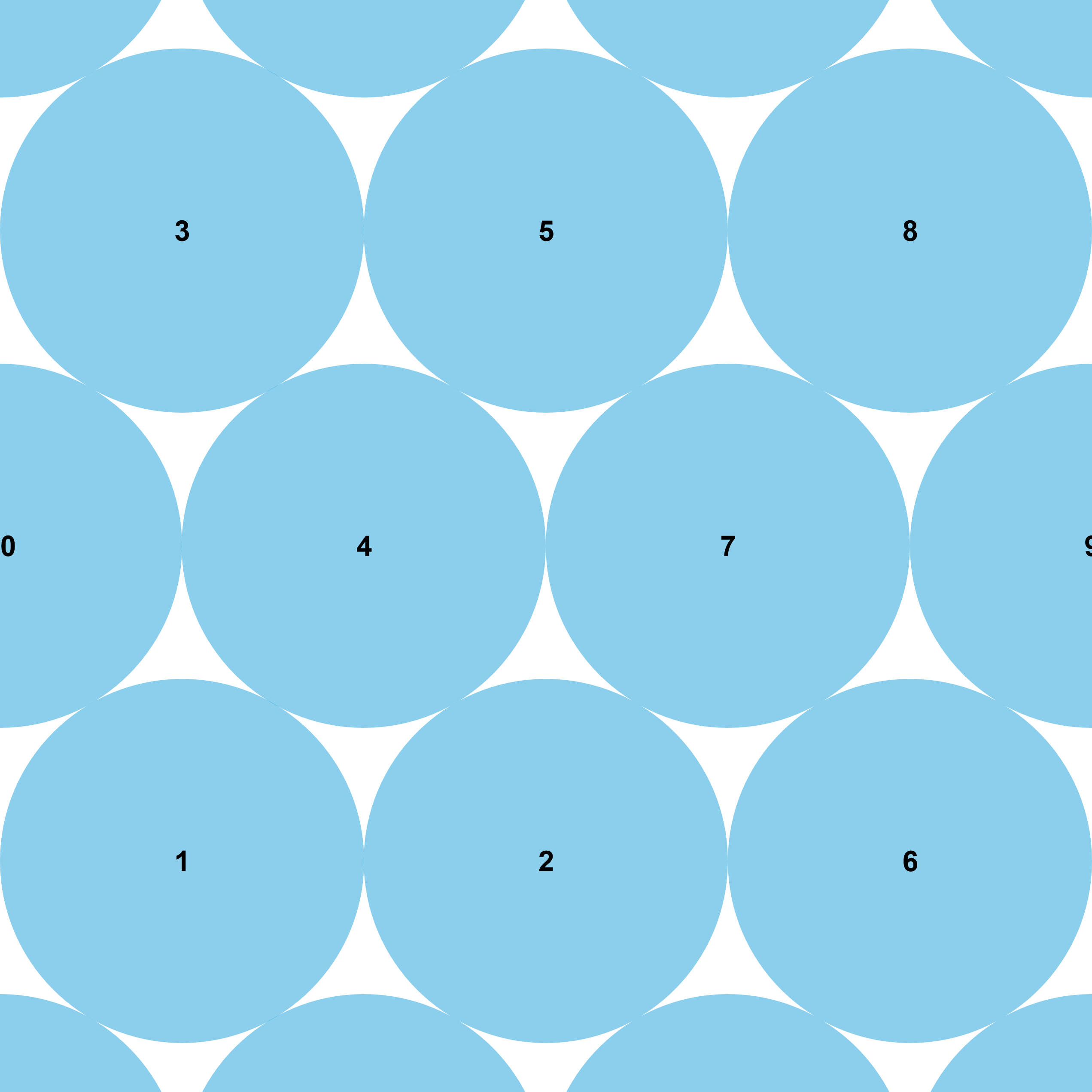}
        \captionsetup{labelsep=colon,margin=1.3cm}
		\caption{The hexagonal packing. It is the densest packing with equal disks in dimension 2}
		\label{hexagonal}
\end{figure}
\begin{figure}
		\centering
		\includegraphics[height=7.5cm]{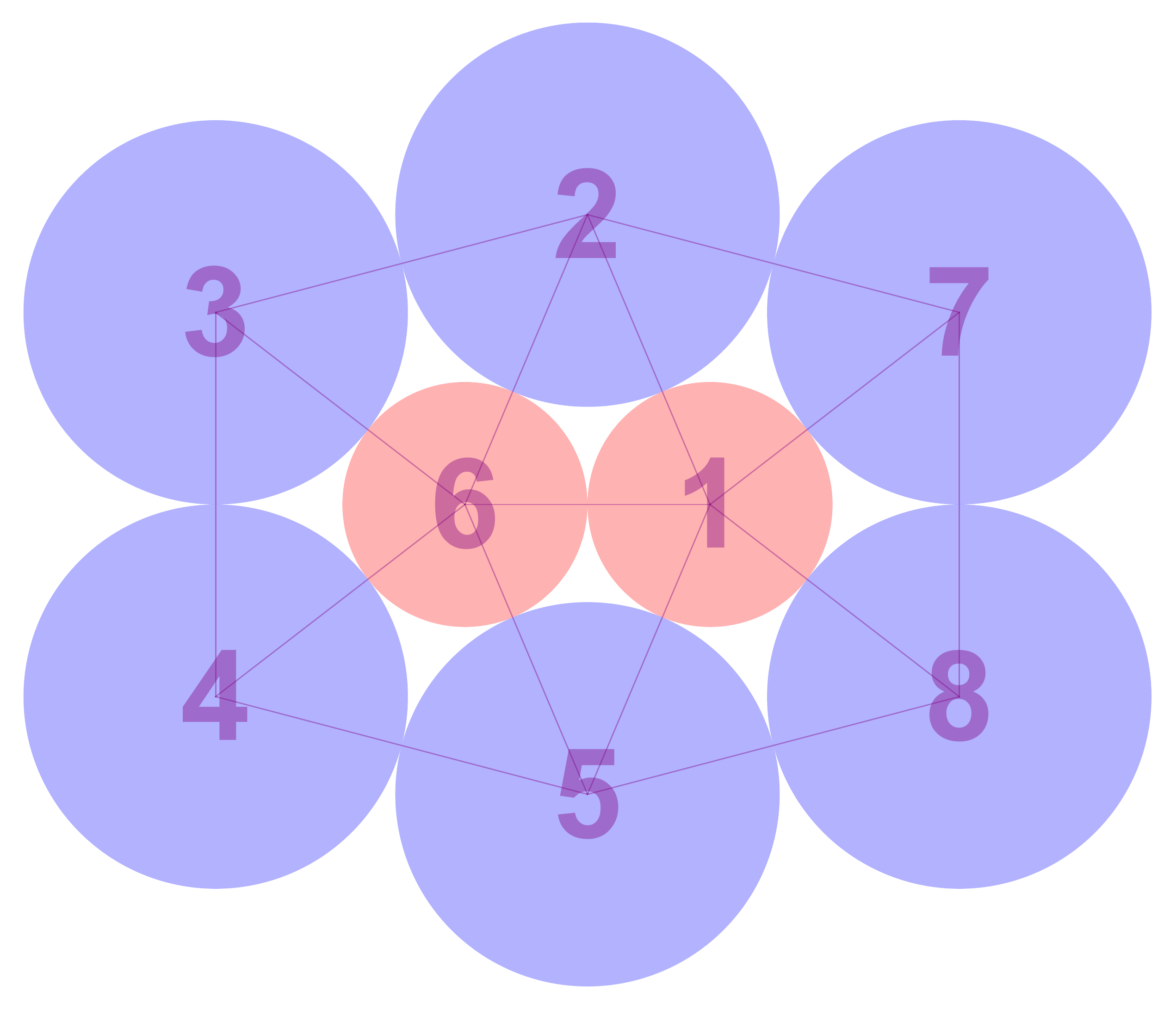}
        \captionsetup{labelsep=colon,margin=1.3cm}
		\caption{An infinite periodic packing with $q\approx 0.6375$ can be constructed by tiling the plane with its boundary hexagon. }
		\label{fejestothhexagon}
\end{figure}

 There is a big gap between 0.6375 and 0.701, but Fejes Tóth's record was undefeated for decades until 2020 when the French computer scientist Thomas Fernique found a better one with $q\approx 0.651$ by enumerating triangulated packings with 3 sizes \cite{Fernique2020}. Part of this packing is given by Figure \ref{fernique}. The gap between $0.701$ and $0.651$ is still quite large. We conjecture that this packing solves the longstanding problem of the largest possible $q$ for triangulated packings with at least two distinct sizes, and that any packing (possibly not triangular) with a density higher than $\frac{\pi}{\sqrt{12}}$ must have $q$ close to $0.651$ (the current record is a packing slightly perturbed from Fernique's, with $q$ approximately $0.658$ \cite{connelly2019maximally}). 
 
\begin{conjecture}\label{uniformconjecture}
    The triangulated periodic packing with 3 sizes demonstrated in Figure \ref{fernique} has the largest radii ratio $q=r_{min}/r_{max}$ among triangulated packings except the hexagonal packing, where the ratio $q$ is a root of $$89x^8+1344x^7+4008x^6-464x^5-2410x^4+176x^3+296x^2-96x+1$$ with a numerical value close to $0.651$\cite{densesttertiary}. 
\end{conjecture}

\begin{figure}
	
		\centering
		\includegraphics[height=7.5cm]{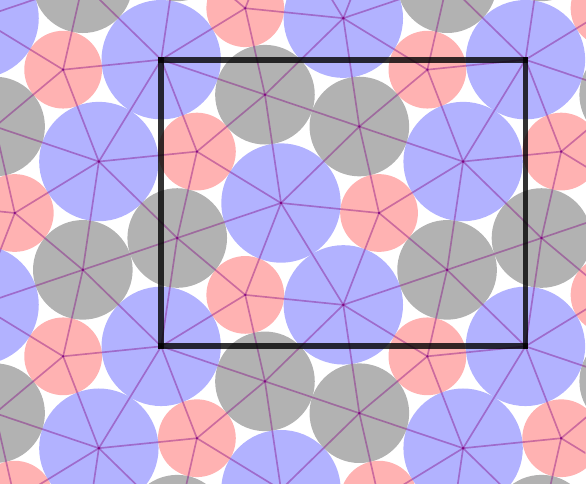}
        \captionsetup{labelsep=colon,margin=1.3cm}
		\caption{A fundamental region is bounded by the black rectangle. This packing has the highest known radii ratio $q$ among any triangulated packing except the hexagonal packing.}
		\label{fernique}
\end{figure}

At first glance, this conjecture seems unreasonable. If Fernique can beat Fejes Tóth's 2 sizes by enumerating 3 sizes, then \textit{surely} it is possible to do better with more sizes. However, having more sizes has a price because the combinatorial structure can force a disk to be larger than another disk that is already large. 

\begin{figure}
	\centering
		\includegraphics[height=7.5cm]{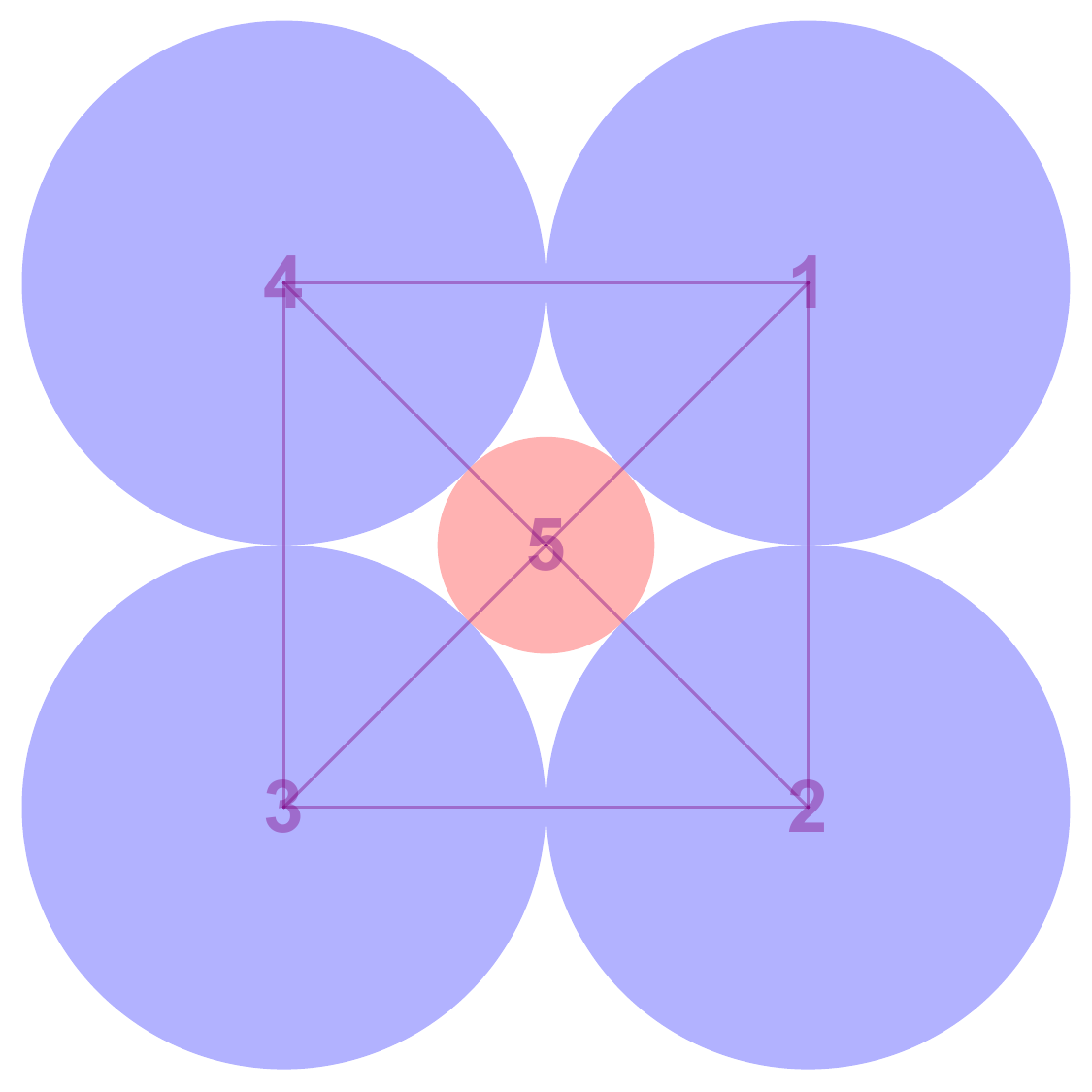}
        \captionsetup{labelsep=colon,margin=1.3cm}
		\caption{a packing representing the graph of a 4-flower. Boundary disks are only allowed to shrink and interior disk are only allowed to expand.}
		\label{4flowers}
\end{figure}

To prove Conjecture \ref{uniformconjecture}, our strategy is to exclude certain finite subgraphs from occurring. Take Figure \ref{4flowers} as an example. If we can prove that it is impossible to increase the size of disk 5 and decrease the size of disk 1 to disk 4 while keeping all the external tangencies, then no disk can have four neighbors in a triangulated packing that beats Fernique's ratio. The reason is that as long as this subgraph occurs, the radii ratio $q\le\frac{\sqrt{2}-1}{1}\approx 0.414$ is smaller than Fernique's $0.651$. 
 
Let $G$ be a finite graph, if there is no finite packing of $G$ with radii in the interval $[q,1]$, then no infinite packing can contain $G$ as a subgraph with radii in $[q,1]$. Thus if a packing of $G$ with radii ratio worse than Figure \ref{fernique} is globally optimal, then $G$ can be eliminated. 

To prove Conjecture \ref{uniformconjecture}, it suffices to prove the triangulation must be the one shown in Figure \ref{fernique} by the uniqueness result of Schramm \cite{Oded}. If we can eliminate sufficiently many finite triangulated graphs so that the triangulation can only be extended as in Figure \ref{fernique}, then Figure \ref{fernique} must have the optimal radii ratio. Hence we have the following observation:
\begin{obsv}
    If Conjecture \ref{uniformconjecture} is true, then it suffices to prove the globally optimal radii ratio is worse than Fernique's for all other triangulations.
\end{obsv}

In fact, if the conjecture is true, we will argue in Section \ref{sectionGlobal} that we only have to compute the globally optimal radii ratio for finitely many finite graphs. This motivates a type of question that has rarely been studied: determining the rigidity of a packing with a given tangency pattern when the radii are bounded by inequality constraints. 
	
        \subsection{Main Results}
Our paper bridges the classical theories of circle packings and the modern geometric rigidity theory of graph embeddings. We (partially) answer the question of when the radii ratio function $q=\frac{r_{\min}}{r_{\max}}$ is locally maximal among circle packings of a given tangency structure. The main result of the first half of the paper is Theorem \ref{firstOrderDuality}. This theorem gives a powerful sufficient condition to assert a packing is locally optimal in terms of the radii ratio $q=\frac{r_{\min}}{r_{\max}}$. The condition, intuitively, says that a circle packing is rigid if and only if there exists an ``internal force'' to stop it from moving infinitesimally without deformation. This theorem is inspired by the analogous result for graphs \cite{rwtf}. The main result of the second half is Theorem \ref{secondOrderDuality}. This theorem, inspired by \cite{secondORigidity}, offers another sufficient test specifically designed to test circle packings which can pass the first test. Intuitively, the second test says the circle packing is rigid if it has an internal force that resists further deformation after the packing has already deformed. These two tests, when combined, almost always work, with no observed exceptions when the packing is triangulated.

	\subsection{Related Results}
In the special case where all the radii are fixed, the problem is equivalent to that of a \textbf{bar framework} in $\mathbb{R}^2$: a graph realization problem with the constraints that all edge lengths are fixed. To see this: make the centers of the disks the set of vertices, and add a bar for each tangent pair. Connelly, Gortler, and Theran proved that if all radii are fixed and generic, then the packing is rigid if and only if there are $2n-3$ touching points \cite{sticky_disc}. 

In the special case where the graph is triangulated, the interior disks are edge connected and that each boundary disk is tangent to an interior disk, Bauer, Stephenson and Wegert proved that the set of packings is a differentiable manifold 
 such that all radii are uniquely determined by the radii on the boundary \cite{bauer_stephenson_wegert_2011}. This result is particularly useful because it gives the global rigidity with certain constraints on the radii. This is one of the few cases where global rigidity can be proved. Collins and Stephenson also give a fast approximation algorithm to compute interior radii from boundary radii \cite{kstep}. 

The question of whether or not a graph has a circle packing with radius in $[q,1]$ is known to be NP-Hard in general. In particular, Breu and Kirkpatrick proved that the problem is NP-complete when $q=1$\cite{BreuKirk}. Therefore, we cannot expect a fast algorithm to determine the global rigidity of circle packings in general. However, the graphs related to Conjecture \ref{uniformconjecture} are all triangulated, and therefore the problem may not be as difficult as the general case.

\subsection{Structure of the Paper}
In Section \ref{sectionDef}, we give our definitions and notation. In Section \ref{sectionInf}, we show the results of the infinitesimal rigidity theory. In Section \ref{sectionComb}, we discuss the combinatorial properties. In Section \ref{sectionGeneric}, we prove rigidity sometimes can be established by the number of algebraically independent radii. In Section \ref{section2ndOrder}, we develop the second-order rigidity theory. In Section \ref{sectionGlobal}, we give examples of globally rigid packings and how they are connected to Conjecture \ref{uniformconjecture}. In Section \ref{sectionDiscussion}, we give the intuitive explanation of different types of rigidity and discuss problems related to finite packings. 
\section{Definitions and Notation}\label{sectionDef}
Let us first formally define the concept of planar graphs, planar embeddings, and circle packings. Let $G=(V,E)$ be an undirected simple graph with $n$ vertices and $m$ edges. 
\begin{defn}
A \textbf{finite circle packing}, $\textbf{p}$, in $\mathbb{R}^2$ is a vector in $\mathbb{R}^{3n}$, $\mathbf{p}:=\langle x_1,y_1,r_1,...,x_n,y_n,r_n\rangle$, where $n$ is the number of disks,  $\textbf{p}_i=(x_i,y_i)$ is the coordinate of the center of the disk $i$, and $r_i>0$ is the radius of the circle $i$. 
\end{defn}

\begin{defn}
    Given an embedding $\textbf{q}$ of $n$ points and a packing $\mathbf{p}$ of $n$ disks in $\mathbb{R}^2$, we say $\mathbf{p}$ is a \textbf{packing of $G$} with orientation $\mathbf{q}$, written as $(G,\mathbf{p})$, if for every edge $e_{ij}\in E$, circle $i$ centered at $\mathbf{p}_i$ and circle $j$ centered at $\mathbf{p}_j$ are externally tangent, i.e. $(r_i+r_j)^2=(x_i-x_j)^2+(y_i-y_j)^2$, and for every vertex in $\mathbf{p}$, its neighbors are embedded in the same order counterclockwise as in $\mathbf{q}$. 
\end{defn}

\begin{defn}
    A packing $\mathbf{p}$ is a \textbf{planar embedding} of graph $G$ if no edges cross each other in the 2-dimensional straight line drawing. A graph that has a planar embedding is called a \textbf{planar graph}. 
\end{defn}

This notation is overloaded in the sense that it is not clear what the orientation is when a packing is written as $(G,\mathbf{p})$. In this paper, all discussions are restricted to one specific orientation, so there is no ambiguity. In particular, when $G$ is topologically a triangulated disk, we take the orientation where no vertex is contained in a triangle. 

 We do not require disks that are not joined by an edge to have disjoint interiors. The reason we do not enforce this common requirement is to avoid global constraints depending both on the combinatorics and the geometry. The orientation is to distinguish different packings of isomorphic graphs. Although the combinatorics are identical for packings of isomorphic graphs, the geometry can be different. In particular, we want to avoid reflections of a subpacking over lines and circles, so the set of packings is a smooth manifold rather than the union of packings for each orientation. 

    By the Koebe–Andreev–Thurston(KAT) Theorem \cite{kobe,andreev,thurston}, a packing of $G$ exists if and only if the graph $G$ is planar. When the graph $G$ is maximal planar (triangulated triangle), then the packing is unique up to Möbius transformations and reflections (for an elementary and constructive proof, see \cite{Flipping_Flow}). Schramm generalized the uniqueness to some classes of infinite triangulated packings \cite{Oded}. 

Among all of the packings, triangulated packings yield simple and interesting results. Almost all packings that were proved to be optimal in terms of density are triangulated. In particular, Figure \ref{fernique} is a finite triangulated packing. We focus on a specific class of triangulated packing:

\begin{defn}A finite packing is \textbf{triangulated} if it's a packing of a graph that is topologically a triangulated disk. A triangulated packing is \textbf{simple} if its interior disks are edge-connected and every boundary disk is adjacent to the interior. 
\end{defn}

	The idea of being ``simple'' is that circle packings of this type behave like harmonic functions on a simply connected region in $\mathbb{R}^2$. It is known that for a simple triangulated packing in the plane, all radii can be determined uniquely and monotonically by the radii on the boundary of the graph (This observation originally comes from Thurston \cite{thurston}, also see\cite{kstep} for details and an algorithm to approximate all radii from boundary radii). Suppose that there are $b$ vertices on the boundary of a triangulated graph $G$, then the set of packings of $G$ is homeomorphic to $(0,\infty)^b$. An elementary proof based on the sums of angles around the interior disks is given in \cite{bauer_stephenson_wegert_2011}. 

To do rigidity on circle packings, let us define what a flex is: 
\begin{defn}
    A \textbf{flex (or motion)} of a packing $(G,\mathbf{p})$ is an analytic path of packings $(G,\mathbf{p}(t))$ in $\mathbb{R}^{3n}$, denoted by $\mathbf{p}(t)=(x_1(t),y_1(t),r_1(t),...,x_n(t),y_n(t),r_n(t))$ that satisfies $(x_j-x_i)^2+(y_j-y_i)^2=(r_i+r_j)^2$ for all edges $e_{ij}$, where $t\in[0,1]$, $\mathbf{p}(0)=\mathbf{p}$. A flex $\mathbf{p}(t)$ is \textbf{trivial} if $\mathbf{p}(t)$ is congruent with $\mathbf{p}(0)$ for all $t\in[0,1]$. 
\end{defn}
\begin{remark}
It would be more natural to define a flex as a continuous path instead of an analytic path. Asimow and Roth showed that these two definitions are equivalent \cite{Asimow1978TheRO,ASIMOW1979171}. 
\end{remark}

By the discussion above, it is not hard to see there is a flex between any two packings(of the same orientation) of a simple triangulated disk if all radii are free. Hence, there must be constraints on the radii in order for rigidity to make sense. The rigidity of packings, in general, is difficult to determine with these constraints. In fact, even rigidity of bar frameworks is difficult to determine beyond the first order. 

\section{Infinitesimal Rigidity}\label{sectionInf}
Let's first define what an infinitesimal flex is:
\begin{defn}
	Given a packing $(G,\mathbf{p})$, a \textbf{infinitesimal flex}(or first order flex) is a vector $\mathbf{p}'=\langle x_1',y_1',r_1',...,x_n',y_n',r_n'\rangle$ satisfying $(\mathbf{p}_i-\mathbf{p}_j)\cdot(\mathbf{p}_i'-\mathbf{p}_j')=(r_i+r_j)(r_i'+r_j')$ for all edges $e_{ij}$. $\mathbf{p}'$ is \textbf{trivial} if it is the derivative of some trivial flex at $t=0$. 
\end{defn}

It is not hard to see why an infinitesimal flex is defined in this way by taking the derivative on both sides of $(r_i+r_j)^2=(x_i-x_j)^2+(y_i-y_j)^2$. This definition is used to ensure that all tangent pairs remain tangent at least infinitesimally. When the packing is more than a single disk, the trivial infinitesimal motions form a 3-dimensional linear space generated by translations and rotations. We assume that our packings have more than one disk in this paper. 

In order to make sense of rigidity and infinitesimal rigidity, constraints are needed on the radii of disks. 
\begin{defn}
    The vertex set $V$ is partitioned into four disjoint sets $V^+, V^-, V^=$, and $V^0$, representing disks that can stay the same or increase their radii, disks that can stay the same or decrease their radii, disks with fixed radii and free disks without constraints, respectively. 
\end{defn}
Throughout the figures in this paper, disks in $V^-$ are colored \textcolor{blue}{blue}, disks in $V^+$ are colored \textcolor{red}{red}, disks in $V^=$ are colored \textcolor{green}{green}, and disks in $V^0$ are colored \textcolor{gray}{gray}. Examples include Figure \ref{fejestothhexagon} and Figure \ref{fernique}. 

Given these constraints, it is possible to define a partial ordering on the set of packings of $G$:

\begin{defn}
	Given $(G,\mathbf{p})$, $(G,\mathbf{\tilde{p}})$ and a vertex partition as above, $\tilde{\mathbf{p}}\succeq \mathbf{p}$ if $\tilde{r}_i\ge r_i$ for all $v_i\in V^+$, $\tilde{r}_i\le r_i$ for all $v_i\in V^-$, and $\tilde{r}_i= r_i$ for all $v_i\in V^=$.
\end{defn}

Three types of rigidity can be defined:
	
	\begin{defn}
	An infinitesimal flex $\mathbf{p}'$ is \textbf{proper} if $r_i'\ge 0$ for every disk $i$ in $V^+$, $r_j'\le 0$ for every disk $j$ in $V^-$, and $r_k'=0$ for every disk $k$ in $V^=$ . A disk packing is \textbf{infinitesimally rigid} (also known as \textbf{first-order rigid}, or \textbf{statically rigid}) if all proper infinitesimal flexes are trivial.
	\end{defn}
	
	\begin{defn}
	A flex $\mathbf{p}(t)$ with $\mathbf{p}(0)=\mathbf{p}$ is \textbf{proper} if $\mathbf{p}(t)\succeq \mathbf{p}$ for $t\in [0,1]$. A packing is (locally) \textbf{rigid} if every proper flex is trivial. 
	\end{defn}
	
	\begin{defn}
	$\mathbf{p}$ is \textbf{globally rigid} if $\tilde{\mathbf{p}}\succeq \mathbf{p}$ implies $\tilde{\mathbf{p}}$ is congruent to $\mathbf{p}$. 
	\end{defn}

    Global rigidity obviously implies rigidity. We will show later that infinitesimal rigidity also implies rigidity. Global rigidity and infinitesimal rigidity do not imply each other. Intuitively, infinitesimal rigidity is a strong form of local rigidity (often preferred by structural engineers). Global rigidity, on the other hand, can be weak at the given packing. 

    Notice that, for global rigidity, we are restricting ourselves to a given orientation to avoid partial reflections. There is no known approach other than brute force to determine the rigidity and global rigidity of a packing in general. The complexity result from \cite{BreuKirk} suggests that no algorithm can determine the general case of global rigidity in polynomial time, unless $P=NP$. Infinitesimal rigidity still yields many interesting results. Next, we will define a matrix that helps us determine infinitesimal rigidity. Similar to the infinitesimal rigidity of tensegrities (the graph embedding problem where there are inequality constraints on the edge lengths), the idea is to define a rigidity matrix such that the kernel is the space of infinitesimal flexes. 

    \begin{defn}
	For a packing $\mathbf{p}$ of a planar graph $G=(V,E)$, where $|V|=n$ and $|E|=m$, we define the \textbf{Rigidity Matrix} $R(\mathbf{p})$ to be the $m\times 3n$ matrix defined as below: 
	 \begin{enumerate}
		\item The $m$ rows are indexed over the edges
		\item The $3n$ columns are indexed over $x_i$, $y_i$, and $r_i$ for $i\in\{1,...,n\}$
		\item For a row indexed by $e_{ij}$, if $k\notin\{i,j\}$, then the entries $x_k,y_k,r_k$ are $0$. 
		\item For a row indexed by $e_{ij}$, if $k=i$, then the entry $x_k$ is $x_i-x_j$, $y_k$ is $y_i-y_j$, $r_k$ is $-r_i-r_j$. If $k=j$, then the entry $x_k$ is $x_j-x_i$, $y_k$ is $y_j-y_i$, $r_k$ is $-r_i-r_j$.
	\end{enumerate} 
	\end{defn}

 Any vector $\mathbf{p}'$ in the kernel of $R(\mathbf{p})$ satisfies $(\mathbf{p}_i-\mathbf{p}_j)(\mathbf{p}_i'-\mathbf{p}_j')=(r_i+r_j)(r_i'+r_j')$ for every edge $e_{ij}\in E$ by the definition of $R(\mathbf{p})$. The infinitesimal rigidity problem is then a linear programming on vectors in the kernel of $R(\mathbf{p})$. However, the kernel necessarily has proper infinitesimal flexes, since the derivative of a congruent motion is always a proper solution. 

 There are four types of proper infinitesimal flexes that can occur in the kernel:
	 \begin{enumerate}
	    \item Trivial infinitesimal flexes are generated by translations and rotations which must fix all radii.
	    \item Proper but non-trivial infinitesimal flexes that do not change any radii.
	    \item Proper infinitesimal flexes that fix $V^+$, $V^=$, and $V^-$ while strictly altering at least one free disk in $V^0$.
	    \item Proper infinitesimal flexes that strictly alter some disks in $V^+$ or $V^-$. 
	\end{enumerate} 

 The first 3 types can be counted by computing the dimension of the kernel with the entries in the corresponding radii being $0$. To force $r_i'=0$ in the kernel, simply add a row to $R(\mathbf{p})$ such that the entry in $r_i$ is $1$ and the other $3n-1$ entries are 0. The last type requires a more sophisticated approach. If the number of constraints is ``just right'', then it might be easier to utilize LP duality to work with the vectors in the cokernel rather than those in the kernel. However, the following proposition \cite[Proposition 2.12]{sticky_disc} states that the cokernel of $R(\mathbf{p})$ has dimension 0: 
	
	\begin{proposition}\label{kerneldimension}
	If $\mathbf{p}$ is a packing of $G$ that has the same orientation with a planar embedding, then the kernel of $R(\mathbf{p})$ has dimension $3n-m$ . 
	\end{proposition}
	
	Intuitively, Proposition \ref{kerneldimension} means that in a packing with the orientation of a planar embedding, no edge is redundant as a linear constraint. Each edge strictly reduces the dimension of the kernel by $1$. Hence, the row-rank is equal to the number of rows and the cokernel has dimension $0$.

	Next, we proceed with the idea of separating the first 3 types of infinitesimal flexes from the last type by forcing some radii to be fixed in the kernel of $R(\mathbf{p})$. Let $e_k$ be the unit row vector in $\mathbb{R}^{3n}$ that is $1$ on the location indexed by $r_k$ and $0$ elsewhere. Let $E_S$ be the matrix with rows $e_k$ for $k\in S$. We can define the new matrix:
	\begin{defn}
	The \textbf{Extended Packing Rigidity Matrix}, $R_e(\mathbf{p})$, of a given configuration $\mathbf{p}$ is defined as the block matrix $\begin{bmatrix}R(\mathbf{p})\\E_{V^-}\\E_{V^+}\\E_{V^=} \end{bmatrix}$. 
	\end{defn}
    This is the ``true'' analog of the rigidity matrix for \textbf{tensegrity framework}: a graph realization problem with inequality constraints on edge lengths (the edges are known as cables, bars, and struts for $\le,=,\ge$ constraints). In the rigidity matrix of a tensegrity framework, the kernel fixes the length of bars, cables, and struts. The extended packing rigidity matrix fixes all radii of disks that have constraints on them. 

    The kernel of $R_e(\mathbf{p})$ has the first 3 types of infinitesimal flexes. If the dimension of the kernel is greater than 3, then either the bar framework is infinitesimally flexible, or the radius of at least one free disk is infinitesimally flexible. To distinguish the second type from the third type, we can further compute the dimension of the kernel for the matrix $R'(\mathbf{p})=\begin{bmatrix}R(\mathbf{p})\\E_{V} \end{bmatrix}$ that fixes all radii. If the kernel of $R'(\mathbf{\mathbf{p}})$ has dimension 3, then it cannot have any nontrivial infinitesimal flex fixing all radii. The difference between the dimension of the kernel of $R'(\mathbf{p})$ and that of $R_e(\mathbf{p})$ tells us if there are infinitesimal flexes that alter the radii of at least one free disk. 

    \textbf{Example}: Consider the packing in Figure \ref{4flowers} where each large disk has radius $1$ and disk 5 is centered at the origin. Suppose that we have a constraint on each disk according to the coloring: disks 1 to 4 can only decrease their sizes, and disk 5 can only increase its size. 

	The extended rigidity matrix is 
	\begin{center}
	\resizebox{1.0\textwidth}{!}{$
	R(\mathbf{p})=\begin{blockarray}{cccccccccccccccc}
	x_1 & y_1& r_1&x_2 & y_2& r_2&x_3 & y_3& r_3&x_4 & y_4& r_4&x_5 & y_5& r_5&\\
	\begin{block}{(ccccccccccccccc)c}
	0 & 2 & -2 & 0 & -2 & -2 & 0 & 0 & 0 & 0 & 0 & 0 & 0 & 0 & 0 & e_{12} \\
	0 & 0 & 0 & 2 & 0 & -2 & -2 & 0 & -2 & 0 & 0 & 0 & 0 & 0 & 0 & e_{23}\\
	0 & 0 & 0 & 0 & 0 & 0 & 0 & -2 & -2 & 0 & 2 & -2 & 0 & 0 & 0 & e_{34}\\
	2 & 0 & -2 & 0 & 0 & 0 & 0 & 0 & 0 & -2 & 0 & -2 & 0 & 0 & 0 & e_{14}\\
	1 & 1 & -\sqrt{2} & 0 & 0 & 0 & 0 & 0 & 0 & 0 & 0 & 0 & -1 & -1 & -\sqrt{2} & e_{15} \\
	0 & 0 & 0 & 1 & -1 & -\sqrt{2} & 0 & 0 & 0 & 0 & 0 & 0 & -1 & 1 & -\sqrt{2} & e_{25} \\
	0 & 0 & 0 & 0 & 0 & 0 & -1 & -1 & -\sqrt{2} & 0 & 0 & 0 & 1 & 1 & -\sqrt{2} & e_{35}\\
	0 & 0 & 0 & 0 & 0 & 0 & 0 & 0 & 0 & -1 & 1 & -\sqrt{2} & 1 & -1 & -\sqrt{2} & e_{45}\\
	0 & 0 & 1 & 0 & 0 & 0 & 0 & 0 & 0 & 0 & 0 & 0 & 0 & 0 & 0 & v_1\\
	0 & 0 & 0 & 0 & 0 & 1 & 0 & 0 & 0 & 0 & 0 & 0 & 0 & 0 & 0 & v_2\\
	0 & 0 & 0 & 0 & 0 & 0 & 0 & 0 & 1 & 0 & 0 & 0 & 0 & 0 & 0 & v_3\\
	0 & 0 & 0 & 0 & 0 & 0 & 0 & 0 & 0 & 0 & 0 & 1 & 0 & 0 & 0 & v_4\\
	0 & 0 & 0 & 0 & 0 & 0 & 0 & 0 & 0 & 0 & 0 & 0 & 0 & 0 & 1 & v_5\\
	\end{block}
	\end{blockarray}
	$
	}
	\end{center}

	In this case, the matrix has a 3-dimensional kernel and a 1-dimensional cokernel. Theorem \ref{firstOrderDuality} on the next page tells us that this packing is infinitesimally rigid since the only vector in its cokernel satisfies some linear constraints. It is useful to first define the concept of stress when talking about the cokernel of $R_e(\mathbf{p})$: 
	
	\begin{defn}
	A \textbf{stress} is a real function $\omega:E\to \mathbb{R}$.  If the net force defined in equation (\ref{equilibrium}) on each vertex is zero vector, then it is an \textit{equilibrium stress}. 
	\end{defn} 
	
	Intuitively, we think of a stress on an edge as a \textquotedblleft force density per length". If the stress on the edge $e_{ij}$ is $\omega_{ij}$, then the force acting on disk $i$ from disk $j$ is the vector $\omega_{ij}(\mathbf{p}_i-\mathbf{p}_j)$ in the direction from $\mathbf{p}_j$ to $\mathbf{p}_i$. The force balance condition at vertex $i$ can be written as a vector sum:
    \begin{equation}\label{equilibrium}
    \sum_{j \vert e_{ij}\in E} \omega_{ij}(\mathbf{p}_i-\mathbf{p}_j)=0
    \end{equation}
	
	For the constraints on the radii, it is helpful to determine whether the forces acting on a disk are pushing the boundary towards its center or pulling it away from its center. This quantity can be calculated as the following force sum:
	\begin{equation}\label{radial}
	\omega_i=\sum_{j\vert e_{ij}\in E} \omega_{ij}(r_i+r_j)
	\end{equation}
	
	The cokernel of $R(\mathbf{p})$ is the equilibrium stresses where the radial force sum defined in equation \ref{radial} is $0$ for every $v_i$. The cokernel of $R_e(\mathbf{p})$ is the equilibrium stresses with the sum of radial force defined in Equation \ref{radial} being $0$ for $v_i\in V^0$. 

	Our main result for the infinitesimal theory is the following theorem showing the equivalence between infinitesimal rigidity and the existence of a specific stress: 
	\begin{thm}\label{firstOrderDuality}
	 A disk packing is infinitesimally rigid if and only if the following conditions hold:
	\begin{enumerate}
		\item{(Fixed Radius Condition)} The packing is infinitesimally rigid when the radii of disks in $V^+$, $V^-$, and $V^=$ are fixed. 
		\item{(Stress Existence Condition)} There exists an equilibrium stress $\omega$ such that the radial force sum defined in equation \ref{radial} is positive on $V^-$, negative on $V^+$, and 0 on $V^0$.
	\end{enumerate}
	\end{thm}
	\begin{proof}
		Intuitively, the first condition eliminates infinitesimal motions that are possible even if you fix all constrained disks. The second condition eliminates the proper infinitesimal motions where some disks in $V^+$ or $V^-$ change their radii properly.
		
		$\Leftarrow$: Let $\omega=(...\omega_{ij}...)$ be the desired stress, and let $\mathbf{p}'$ be any proper infinitesimal motion. Extend our $\omega$ to $\omega_e=(...\omega_{ij}...\omega_{k}...)$, so it lives in the cokernel of $R_{e}(\mathbf{p})$, where $\omega_{ij}$ are indexed on edges and $\omega_k$ is indexed over disks in $V^+$, $V^-$ and $V^=$. 
		
		For any proper infinitesimal motion $\mathbf{p}'$, we have
		\begin{equation}\label{3product}
		    \omega_e R_e(\mathbf{p})\mathbf{p}'=\sum_{e_{ij}}\omega_{ij}((\mathbf{p}_i-\mathbf{p}_j)\cdot(\mathbf{p}_i'-\mathbf{p}_j')-(r_i+r_j)(r_i'+r_j'))+\sum_{k\notin V^{0}}\omega_kr_k'
		\end{equation}
		
		The first part is $0$ by the definition of an infinitesimal motion. The second term is non-positive because we assumed $\omega_k$ and $r_k'$ either have opposite signs or $r_k'=0$. Hence any $\omega_kr_k'\not=0$ would imply $\omega_e R_e(\mathbf{p})\mathbf{p}'<0$, contradicting $\omega_e$ being in the cokernel of $R_e(\mathbf{p})$. Therefore, all $r_k'=0$, and by Condition 1, the packing is infinitesimally rigid. 
		
		$\Rightarrow$: If a packing is infinitesimally rigid, then (i) is automatically true. To prove (ii), we need the following lemma. The notation $\textbf{v}>0$ for a vector $\textbf{v}$ means the inequality holds for every coordinates. 

        \begin{lem}\label{farkas}
		
		(Farkas'  Alternative): Let $A\in\mathbb{R}^{m\times n} $ and $b\in\mathbb{R}^m$. Either $Ax=b$ has a solution and $x\ge 0$, or there exists $y$ such that $A^T y\le 0$ but $y^Tb>0$. 
    \end{lem}
		
		The idea is to construct a matrix $A$ such that $A^Ty\le 0$ holds if and only if $y$ is a proper infinitesimal motion. Then, $y^T b\not=0$ suggests that the infinitesimal motion is non-trivial. 
		
		Let $A^T=\begin{bmatrix}
		R(\mathbf{p})\\
		-R(\mathbf{p})\\
		-E_{V^+}\\
		E_{V^-}\\
		E_{V^=}\\
		-E_{V^=}
		\end{bmatrix}$. Let $b$ be any vector in $\mathbb{R}^{3n}$ positive on radii of $V^+$, negative on radii of $V^-$, and 0 everywhere else. $A^Ty\le 0$ forces $y$ to be an infinitesimal motion because $R(\mathbf{p})y\le0$ and $-R(\mathbf{p})y\le 0$ implies $R(\mathbf{p})y=0$. $y$ is a proper infinitesimal motion because all the entries of $y$ indexed by the radii have the correct signs. $y^Tb>0$ implies the infinitesimal motion $y$ cannot fix all radii. 
		
		Since our packing is infinitesimally rigid, such a $y$ must not exist. As a result, the other case $Ax=b$ and $x\ge 0$ must be true. Suppose $x=[\omega^+,\omega^-,\omega^{Ex}]$ is such a solution where the 3 components correspond to $R(\mathbf{p})$,$-R(\mathbf{p})$, and the remaining rows. Let $\omega_i$ be the entry in $\omega^{Ex}$ corresponding to the radius of disk $i$ for $v_i$ in $V^+$, $V^-$ or $V^=$. 
		
		First, observe $\omega=\omega^+-\omega^-$ is an equilibrium stress on the edges because the $x,y$ terms in $Ax=b$ satisfy:
		\begin{equation*}
		    \sum_{j\vert ij\in E}(\omega^+_{ij}-\omega^-_{ij})(\mathbf{p}_i-\mathbf{p}_j)=0
		\end{equation*}

		The radius terms, using $v_i\in V^-$ as an example, in $Ax=b$ has the form
		\begin{equation*}
		    \sum_{j\vert ij\in E}(\omega^+_{ij}-\omega^-_{ij})(-r_i-r_j)+\omega_i=b_i
		\end{equation*}
		
		It can be re-written as
		\begin{equation*}
		\sum_{j\vert ij\in E}(\omega^+_{ij}-\omega^-_{ij})(r_i+r_j)=-b_i+\omega_i
		\end{equation*}
		
		Since $i\in V^-$, $b_i< 0$. By our assumption of $x\ge 0$, we have $\omega_i\ge 0$. As a result, the radial sum of forces around $i$ for stress $\omega^+-\omega^-$ satisfies
		\begin{equation*}
		\sum_{j\vert ij\in E}(\omega^+_{ij}-\omega^-_{ij})(r_i+r_j)>0
		\end{equation*}
		
		For $i\in V^+$, $\omega_i$ is replaced by $-\omega_i$ and $b_i$ is positive, the proof is then identical. For $i\in V^0$, the positive and negative terms of $\omega_i$ cancel and $b_i=0$. 
	\end{proof}

\begin{figure}
	\centering
		\includegraphics[height=7.5cm]{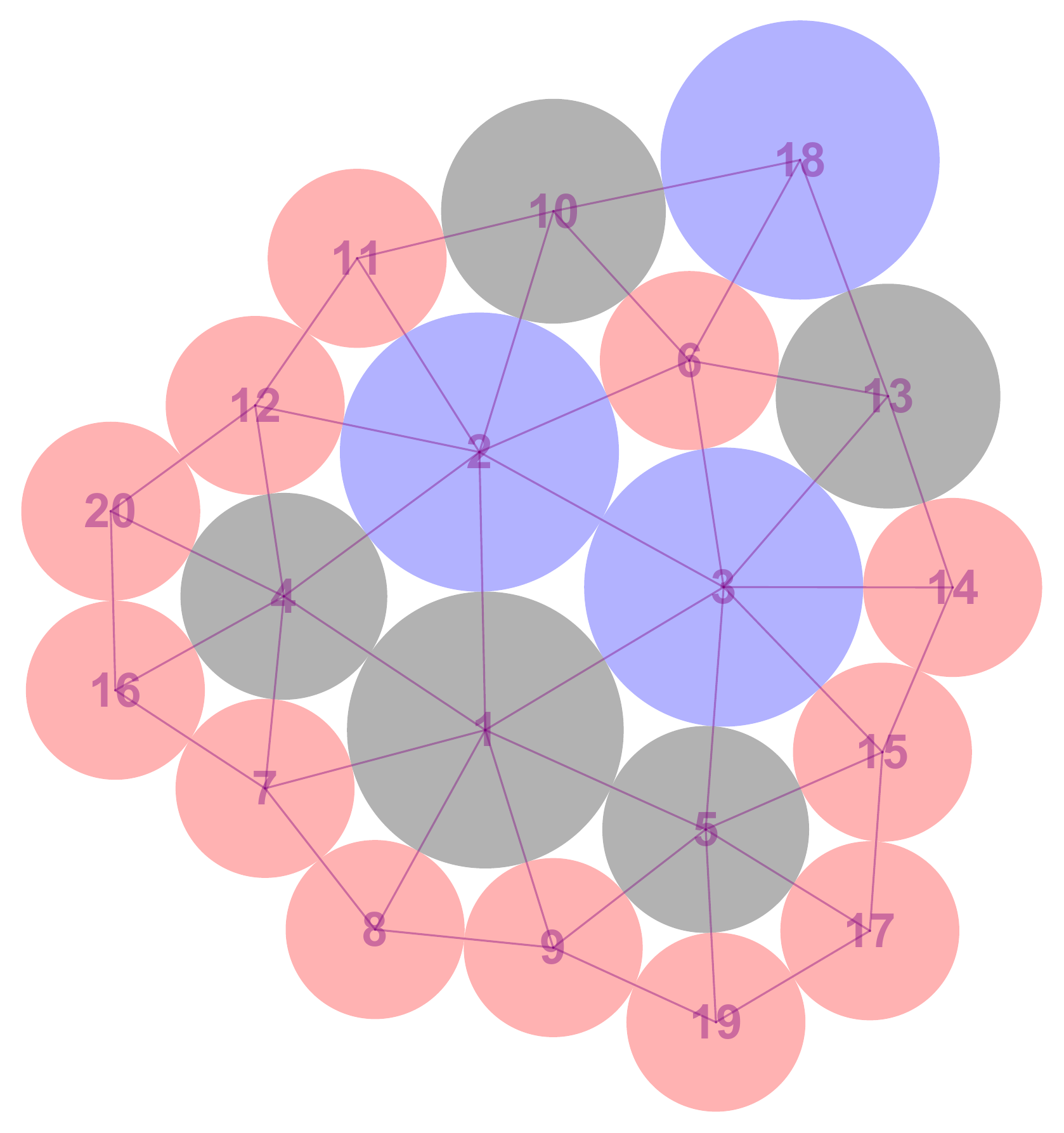}
        \captionsetup{labelsep=colon,margin=1.3cm}
		\caption{A packing that is infinitesimally rigid by the main theorem \ref{firstOrderDuality}. The red disks all have the same size, and the same is true for blue disks. It is impossible to increase the size of red disks while decreasing the size of blue disks locally.}
		\label{20disk}
\end{figure}

    Figures \ref{fejestothhexagon},  \ref{4flowers},  and  \ref{20disk} can be proved to be infinitesimally rigid using this theorem. In fact, a stronger statement holds even when we no longer insist that the disks stay tangent. The following corollary is obvious from the proof above:\begin{corollary}
	
	If the stress existence condition holds with stress $\omega$, then any proper(in the sense of radii change) infinitesimal flex $\mathbf{p}'$ ignoring tangency conditions while satisfying the following two conditions must preserve all radii in $V^+\cup V^-$ and all tangent relations between disk $i$ and disk $j$ if $\omega_{ij}\not=0$:
	 \begin{enumerate}
	    \item If $\omega_{ij}<0$, disk $i$ and disk $j$ remain tangent or become separated 
	    \item If $\omega_{ij}>0$, disk $i$ and disk $j$ remain tangent or become overlapped 
	\end{enumerate}
	\end{corollary}
	\begin{proof}
	    Consider $$\omega R_e(\mathbf{p})\mathbf{p}'=\sum_{e_{ij}}\omega_{ij}((\mathbf{p}_i-\mathbf{p}_j)\cdot(\mathbf{p}_i'-\mathbf{p}_j')-(r_i+r_j)(r_i'+r_j'))+\sum_{k\notin V^0}\omega_kr_k'$$ If the above conditions hold, then each term is non-positive. However, $\omega$ is in the cokernel of $R_e(\mathbf{p})$, and hence all terms are $0$. 
	\end{proof}

	The reason we study infinitesimal rigidity is that it implies rigidity. This is a consequence of the product rule for derivatives:
	\begin{proposition}\label{firstorderrig}
	If a disk packing is infinitesimally rigid, it is rigid. 
	\end{proposition} 
	\begin{proof}
		If a packing is infinitesimally rigid, then there is no non-trivial proper infinitesimal motion $\mathbf{p}'$ that satisfies $$(\mathbf{p}_i-\mathbf{p}_j)\cdot(\mathbf{p}_i'-\mathbf{p}_j')-(r_i+r_j)(r_i'+r_j')=0$$ 
		
		Let $\mathbf{p}(t)$ be a flex and $n$ be the smallest number such that the $n^{th}$ derivative of $\mathbf{p}(t)$ is nontrivial. Consider the $n^{th}$ derivative of the tangency condition, 
		\begin{equation}
		\begin{split}0=\left(\frac{\partial}{\partial t}\right)^n[(\mathbf{p}_i(t)-\mathbf{p}_j(t))\cdot (\mathbf{p}_i(t)-\mathbf{p}_j(t))-(r_i(t)+r_j(t))(r_i(t)+r_j(t))]\\
		=\sum_{i=0}^n \binom{n}{i}\left(\frac{\partial}{\partial t}\right)^i(\mathbf{p}_i(t)-\mathbf{p}_j(t))\cdot \left(\frac{\partial}{\partial t}\right)^{n-i}\left(\mathbf{p}_i(t)-\mathbf{p}_j(t)\right)\\-\left(\frac{\partial}{\partial t}\right)^i(r_i(t)+r_j(t))\left(\frac{\partial}{\partial t}\right)^{n-i}(r_i(t)+r_j(t))
		\end{split}
		\end{equation}
		
		Without loss of generality, we assume that the first $n-1$ derivatives of $\mathbf{p}(t)$ are all zeroes. Now, if we plug in these $0$s, then we are left with:
		$$(\mathbf{p}_i(t)-\mathbf{p}_j(t))\cdot \left(\frac{\partial}{\partial t}\right)^{n}(\mathbf{p}_i(t)-\mathbf{p}_j(t))-(r_i(t)+r_j(t))\left(\frac{\partial}{\partial t}\right)^{n}(r_i(t)+r_j(t))=0$$
		
		Observe that the $n^{th}$ derivative $\left(\frac{\partial}{\partial t} \right)^n \mathbf{p}(t)$ satisfies the same expression as the one we need for infinitesimal rigidity. Therefore, by induction, if there is no proper non-trivial infinitesimal flex, there cannot be any non-trivial proper flex on its $n^{th}$ derivative. Therefore, there cannot be a proper nontrivial analytical path at $t=0$. 
	\end{proof}

 	The proof of Theorem \ref{firstOrderDuality} does not depend on the dimension, so analogous results are true for sphere packings of any dimension. In particular, we can determine whether or not a 3D ball packing is infinitesimally rigid given its combinatorial structure. For example, if 6 unit spheres form a regular octahedron with a small sphere of radius $\sqrt{2}-1$ inserted in the center, then Theorem \ref{firstOrderDuality} can show this packing is rigid under the constraint that the 6 unit spheres are in $V^+$ and the central small sphere is in $V^-$. 
	
	Now we know that infinitesimal rigidity implies rigidity. Is the reverse true? The answer is negative. It is possible for a packing to be rigid locally with a non-trivial infinitesimal flex. In these cases, infinitesimal rigidity does not give us information on whether or not the structure is rigid. One such example is the packing in Figure \ref{10diskprestress}, where the disks of identical color have the same size.
	\begin{figure}
		\centering
		\includegraphics[height=7.5cm]{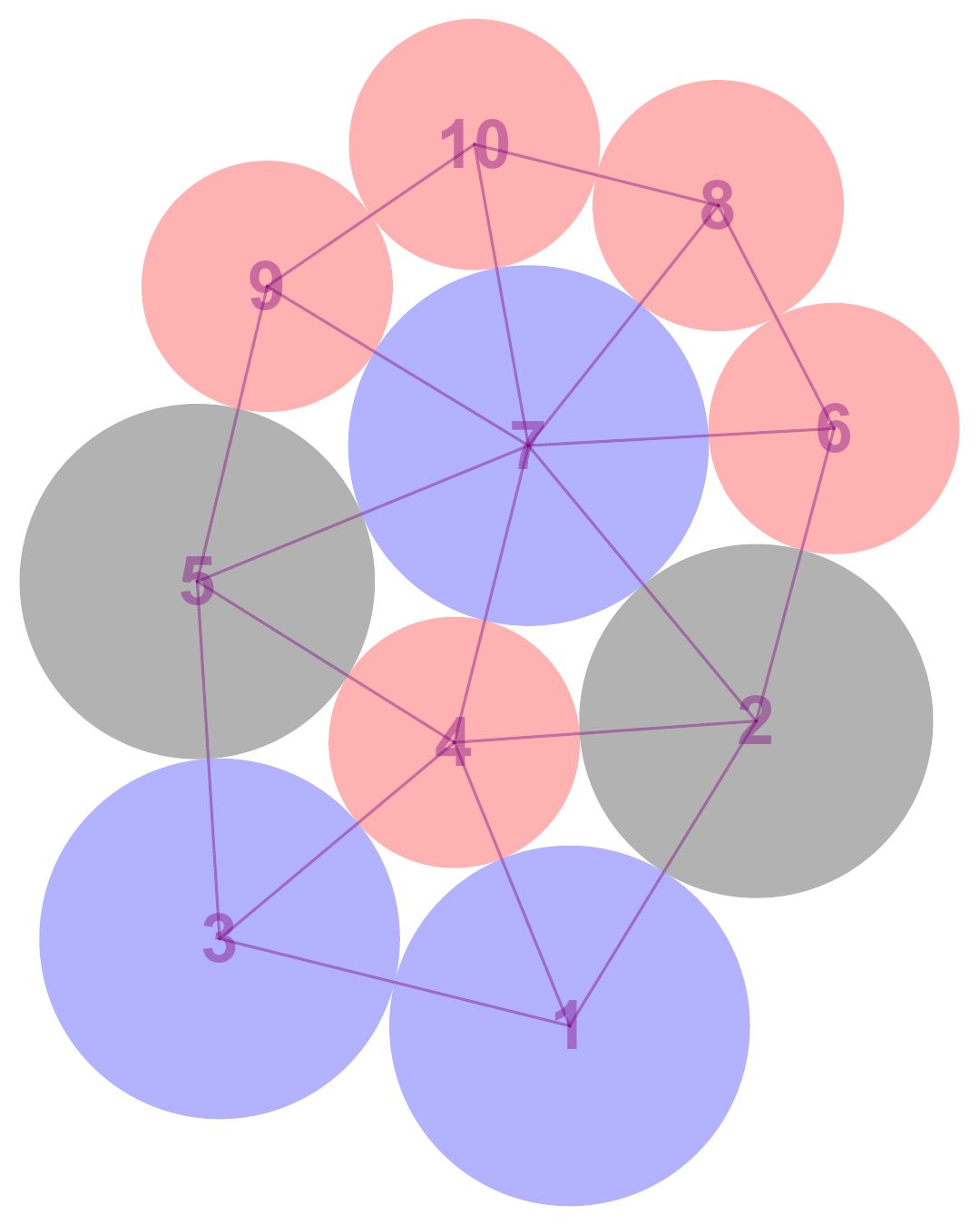}
        \captionsetup{labelsep=colon,margin=1.3cm}
		\caption{a packing that is infinitesimally flexible, but rigid}
		\label{10diskprestress}
	\end{figure}
    \begin{subsection}{A case study of a globally rigid packing}
	We can use some of the ideas from \cite{bauer_stephenson_wegert_2011} to analyze Figure \ref{10diskprestress}. If we use the boundary radii as a basis to determine interior radii, then due to symmetry disks 2 and 5 must have identical partial derivatives with respect to the interior. This means increasing the size of disk 2, decreasing the size of disk 5 at the same rate, and keeping all other radii fixed should be part of a valid infinitesimal motion. This can be easily verified by computing the dimension of the null space of the extended rigidity matrix of this packing, which is 4. 
 
	There are essentially two constraints that are determined only by the radii - the angle sums around disk 4 and 7 must be $2\pi$. We define the following function that maps 3 radii of a triangle to the angle:
	\begin{equation*}
	    \alpha(x,y,z)=\cos ^{-1}\left(\frac{(x+y)^2+(x+z)^2-(y+z)^2}{2 (x+y) (x+z)}\right)
	\end{equation*}
	\begin{figure}
		\centering
		\includegraphics[scale=0.4]{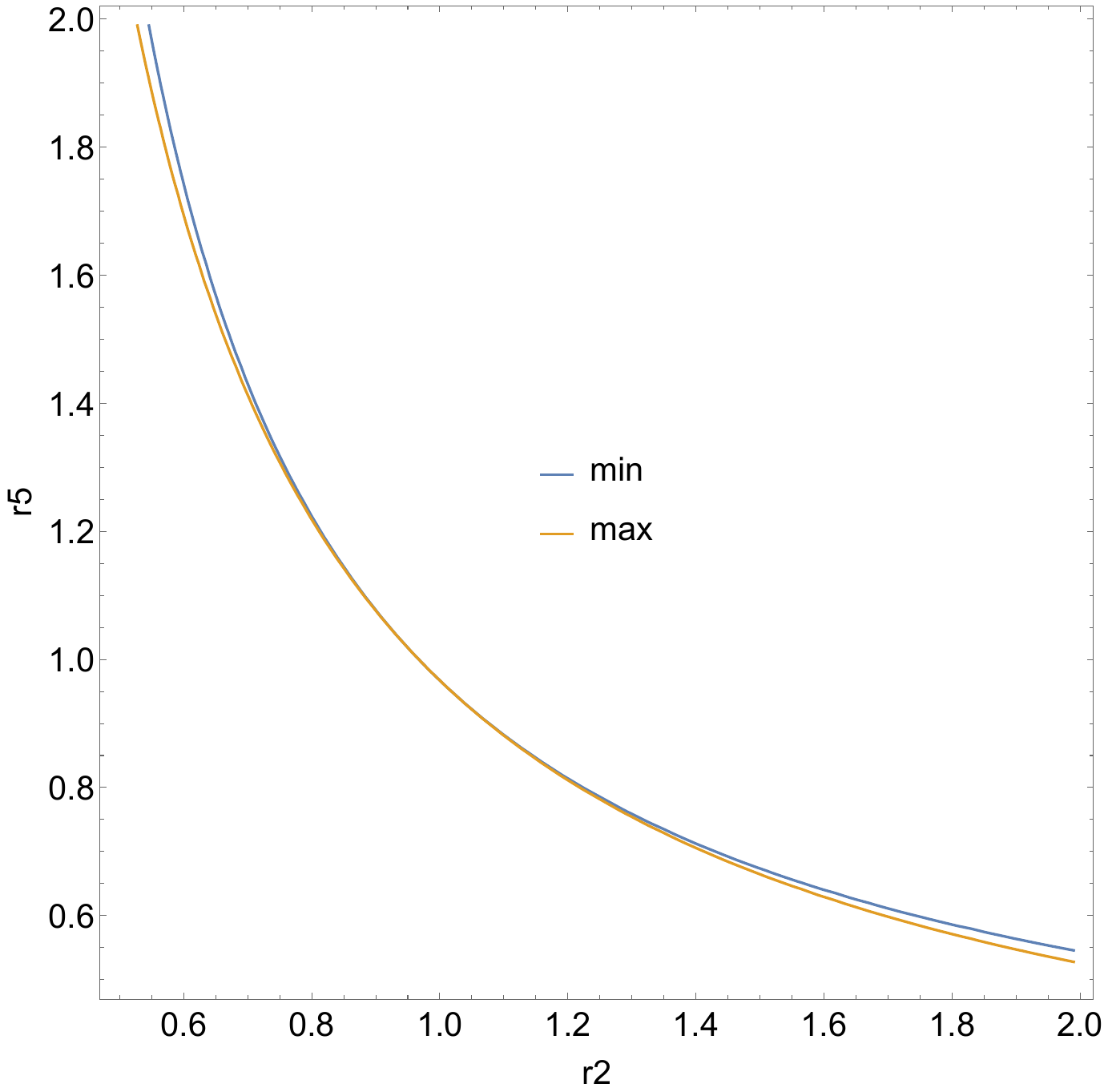}
        \captionsetup{labelsep=colon,margin=1.3cm}
		\caption{The contour plot for minimal and maximal $r_5$ at given $r_2$}
		\label{minmax}
	\end{figure}
	It is well known $\alpha(x,y,z)$ is monotonically increasing with $y$ and $z$ while monotonically decreasing with $x$\cite{kstep}. Using these monotonic relations, we can bound one of $r_2$ and $r_5$ given the other. Consider all possible packings where blue disks are not bigger and red disks are not smaller. Fixing the radius of disk 2, the radius of disk 5 is minimal when the radii of disks 1,3,4, and 7 are fixed (otherwise, the angles around disk $7$ cannot sum up to $2\pi$ due to monotonicity previously mentioned). Similarly, the radius of disk 5 is maximal when the radii of disks 4,7,6,8,9, and 10 are fixed. As a result, we can compute the minimal and maximal value of $r_5$ given $r_2$. The minimal value of $r_5$ satisfies the following.
	\begin{equation*}
	    \alpha(r_4,1,1)+2\alpha(r_4,r_5,1)+2\alpha(r_4,r_2,1)=2\pi
	\end{equation*}
	Similarly, the maximal value of $r_5$ must satisfy:
	\begin{equation*}
	    3\alpha(1,r_4,r_4)+2\alpha(1,r_4,r_2)+2\alpha(1,r_4,r_5)=2\pi
	\end{equation*}
	The two contours are shown in Figure \ref{minmax}.
	Denote the radius $r_4$ as $q$. Solving these equations, we have the following lower and upper bounds for $r_5$:
	\begin{equation}\label{minr}
	   \min r_5= \frac{-2 q (q+1) \cos ^2\left(\frac{1}{4} \left(\cos ^{-1}\left(\frac{q^2+2 q-1}{(q+1)^2}\right)+2 \cos ^{-1}\left(\frac{(q-1) r_2+q (q+1)}{(q+1) \left(q+r_2\right)}\right)\right)\right)}{(q+1) \cos \left(\frac{1}{2} \cos ^{-1}\left(\frac{q^2+2 q-1}{(q+1)^2}\right)+\cos ^{-1}\left(\frac{(q-1) r_2+q (q+1)}{(q+1) \left(q+r_2\right)}\right)\right)+q-1}
	\end{equation}
	\begin{equation}\label{maxr}
	   \max r_5= \frac{-2 (q+1) \cos ^2\left(\frac{1}{4} \left(3 \cos ^{-1}\left(\frac{-q^2+2 q+1}{(q+1)^2}\right)+2 \cos ^{-1}\left(\frac{q \left(-r_2\right)+q+r_2+1}{q r_2+q+r_2+1}\right)\right)\right)}{(q+1) \cos \left(\frac{3}{2} \cos ^{-1}\left(\frac{-q^2+2 q+1}{(q+1)^2}\right)+\cos ^{-1}\left(\frac{q \left(-r_2\right)+q+r_2+1}{q r_2+q+r_2+1}\right)\right)-q+1}
	\end{equation}
	\begin{figure}
		\centering
		\includegraphics[height=7cm]{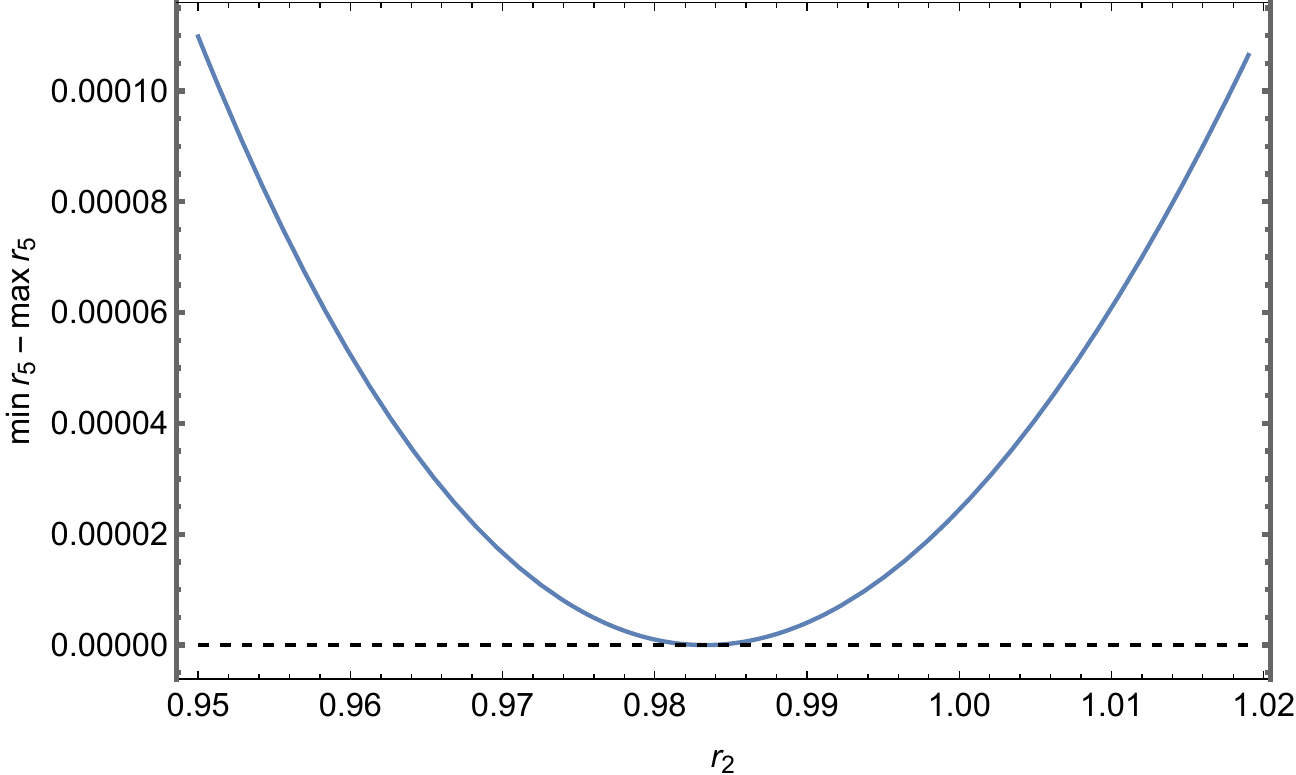}
        \captionsetup{labelsep=colon,margin=1.3cm}
		\caption{The value for $\min r_5-\max r_5$ at given $r_2$}
		\label{feasibleinterval}
	\end{figure}
Figure \ref{feasibleinterval} gives a closer view of $(\min r_5-\max r_5)$ near the intersection point. We can see that the interval $[r_{min},r_{max}]$ is not empty at one point when the two contours intersect in the configuration in Figure \ref{minmax}. The exact values of $q=r_4$ and $r_2$ at the intersection point are as follows. $r_4$ is a real root approximately $0.697$ for
  \begin{align*}
      &529 x^{20}-10812 x^{19}+67214 x^{18}-178684 x^{17}+402509 x^{16}-425904 x^{15}\\&-234840 x^{14}-949488 x^{13}-181390 x^{12}+2943928 x^{11}+2762772 x^{10}\\&-1255176 x^9-2284014 x^8-317040 x^7+527784 x^6+216016 x^5\\&-6915 x^4-23388 x^3-5362 x^2+548 x+289
  \end{align*} 
  $r_2$ is a real root of $y$ approximately $0.983$ for 
  \begin{align*}
      &16 x^6 y-x^6+48 x^5 y^2+44 x^5 y-2 x^5+48 x^4 y^3+58 x^4 y^2+40 x^4 y-x^4\\&+16 x^3 y^4-20 x^3 y^3-28 x^3 y^2+12 x^3 y-33 x^2 y^4-56 x^2 y^3-38 x^2 y^2\\&+14 x y^4+12 x y^3-y^4
  \end{align*}
  where $x$ is the $r_4$ value above. 

    From Equation \ref{minr}, when $r_2\le0.96$, $r_5$ must be larger than $1$ so we only have to consider the packings where $0.96<r_2\le1$. After plugging in the value for the radii given by Figure \ref{10diskprestress} into $(\min r_5-\max r_5)$, the function value and its first derivatives is $0$ at the given value (computed symbolically with Mathematica) while the second derivative is positive on interval $0.96<r_2<1$ (verified numerically using Mathematica with arbitrary precision so that the output error is within $10^{-16}$. The smallest value of $\frac{d^2}{dr_2^2}((\min r_5-\max r_5))$ is bigger than $0.47$). Therefore, the function $(\min r_5-\max r_5)$ is convex on $0.96<r_2<1$ with a critical zero point, so the zero point is unique. Having exactly one nonpositive point means that Figure \ref{10diskprestress} is globally rigid. These calculations give the following result: 
    \begin{proposition}
        The packing shown in Figure \ref{10dgen} has globally maximal $q=\frac{r_{min}}{r_{max}}$ among any packing of its graph. 
    \end{proposition}
    These calculations are done with Wolfram Mathematica 14.0\cite{WMathematica}. 
    \end{subsection}
\section{The Combinatorics}\label{sectionComb}
    This section discusses the combinatorics of circle packing so we can discuss generic circle packings later. 
	\subsection{The dimension count}
	If a planar packing $(G,\mathbf{p})$ has $n$ vertices and $m$ edges, by Proposition \ref{kerneldimension} the kernel of $R(\mathbf{p})$ always has dimension $3n-m$\cite{sticky_disc}. The trivial motions generate a tangent space of dimension 3 unless there is only a single disk. In order to make a packing infinitesimally rigid, we need to reduce the degree of freedom by $3n-m-3$. Since we need $k$ equalities or $k+1$ inequalities to reduce the degree of freedom by $k$, we need at least $3n-m-3$ fixed disks, or $3n-m-2$ constraints including inequalities, to make a packing infinitesimally rigid. 
	
	In the special case where the packing is simple and triangulated, we can use Euler's formula to show 
	\begin{equation}\label{boundary}
	    3n-m-3=b
	\end{equation} where $b$ is the number of disks on the boundary.  Euler's characteristic gives 
	\begin{equation}\label{euler}
	    n-m+f=1
	\end{equation} where $f$ is the number of triangular faces. Other than the boundary edges, every other edge is used on two triangular faces. This gives
	\begin{equation}\label{triangles}
	    3f=2m-b
	\end{equation} Substituting equation (\ref{triangles}) into (\ref{euler}) gives (\ref{boundary}). This is consistent with the result proved in \cite{bauer_stephenson_wegert_2011} that such a packing is globally uniquely determined by the boundary radii. In Figure \ref{10diskprestress}, there are 8 disks on the boundary and 8 inequality constraints. Therefore, it is clear that Figure \ref{10diskprestress} cannot be infinitesimally rigid simply by counting.
	
	In summary, the infinitesimal degree of freedom (dimension of the kernel of $R(\mathbf{p})$) is always $\dim\ker(R(\mathbf{p}))-3$. If the packing is simple and triangulated, then this number is equal to the number of disks on the boundary $b$. For infinitesimal rigidity, at least $b$ equality or $b+1$ inequality constraints are needed. In the next subsection, we determine some combinatorial properties of the set of disks that infinitesimally rigidify a given packing. 
	
	\subsection{Maximal independent set}
	To understand the structure of disks that can rigidify a given packing, the concept of an independent set can be helpful. Intuitively, if the radius $r_i$ is determined by the radii in some set of disks $S$, then $r_i$ does not provide more information once all the radii in $S$ are known. Since we focus on the set of disks that rigidify the packing infinitesimally, linear independence makes sense. 
	\begin{defn}
	Given a packing $(G,\mathbf{p})$, a set $S$ of disks is \textbf{linearly independent} if the row rank of $\begin{bmatrix}R(\mathbf{p})\\ E_S\end{bmatrix}$ is $m+|S|$, where $m$ is the number of edges. Otherwise, we say $S$ is linearly dependent. An independent set $S$ is \textbf{maximal} if there does not exist a disk $d_i$ such that $S\cup\{d_i\}$ is linearly independent. 
	\end{defn}
	Intuitively, if a subset is linearly dependent, then some radius is determined by other radii infinitesimally. To see how this definition of independence works, suppose that the row rank is not maximal; there exists $e_i$ where $i\in S$ is in the row space of $\begin{bmatrix}R(\mathbf{p})\\ E_{S-\{i\}}\end{bmatrix}$, then there is a vector $\omega$ in the left null space of $\begin{bmatrix}R(\mathbf{p})\\ E_S\end{bmatrix}$. According to Equation \ref{3product}, we can assign $V^+$,$V^-$, and $V^=$ based on the signs of $\omega$ so that no infinitesimal flex can vary the radii of these disks in a proper way. However, since we do not have the fixed radius condition in Theorem \ref{firstOrderDuality}, the packing may not be rigid. In this case, only the radii in the subset $S$ are guaranteed to be infinitesimally rigid while the rest of the packing can flex. 
	 
	 \begin{figure}
	     \centering
	    \includegraphics[height=7.5cm]{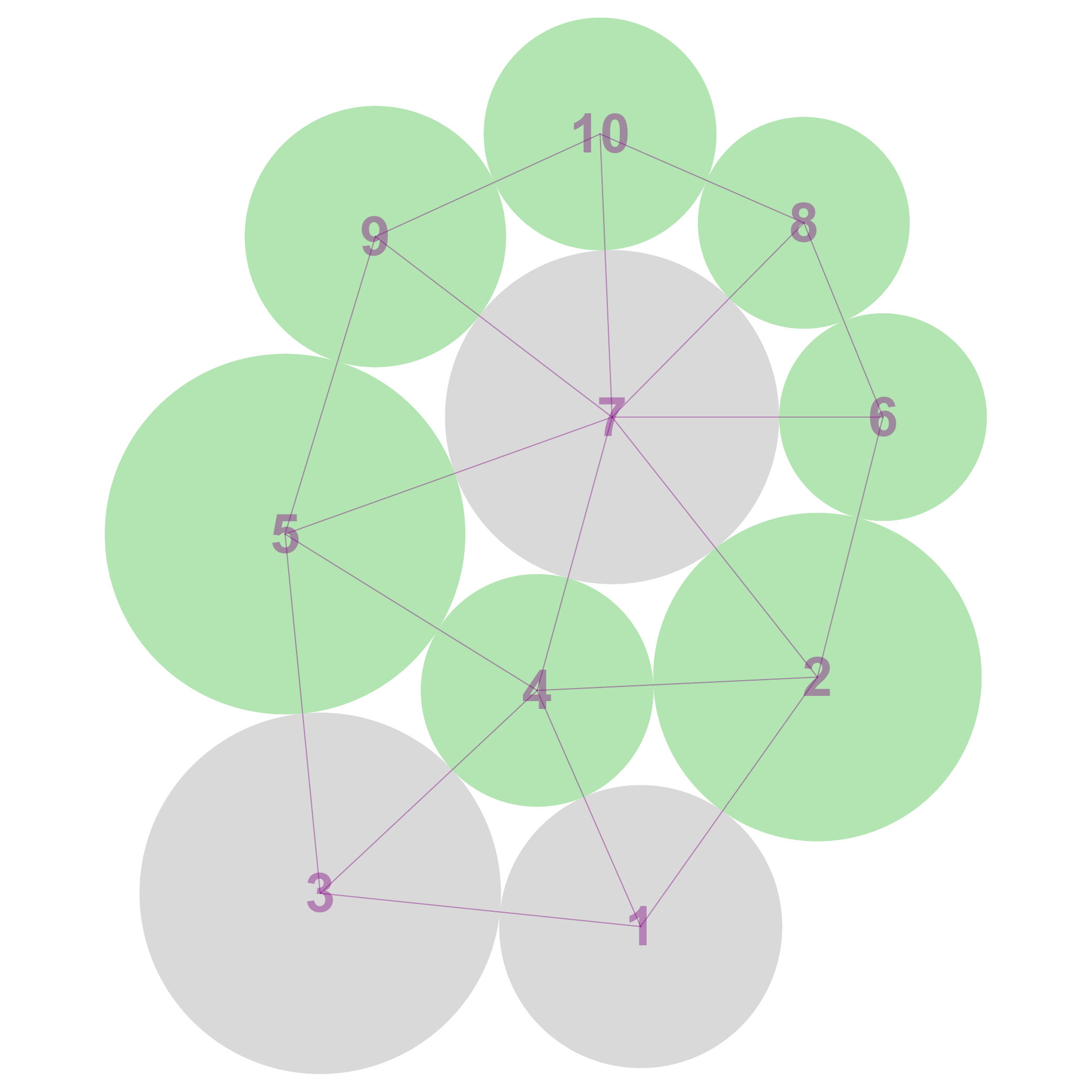}
        \captionsetup{labelsep=colon,margin=1.3cm}
		\caption{The green disks in \{2,4,5,6,8,9,10\} form a linearly independent set that determines $r_7$ infinitesimally. This does not make it infinitesimally rigid, as disk $1$ and $3$ can still vary their sizes.}
	     \label{10dgen}
	 \end{figure}
	
	An independent set $S$ is maximal if the radii in $S$ determine all the radii infinitesimally, that is, $\begin{bmatrix}R(\mathbf{p})\\ E_S\end{bmatrix}$ and $\begin{bmatrix}R(\mathbf{p})\\ E_V\end{bmatrix}$ have the same rank where $V$ is the whole vertex set. Observe that $\begin{bmatrix}R(\mathbf{p})\\ E_V\end{bmatrix}$ has the same kernel as the submatrix produced by removing all columns corresponding to the radii. Hence, the kernel is the set of infinitesimal motions of the bar framework. An example is given in Figure \ref{10dgen}. The set $S$ of green disks is linearly independent but not maximal. $S\cup \{7\}$ is neither maximal nor linearly independent. $S\cup \{1\}$ is both maximal and linearly independent. 
	
	Next, we focus on the packings with an infinitesimally rigid bar framework. In these packings $\begin{bmatrix}R(\mathbf{p})\\ E_V\end{bmatrix}$ has a 3-dimensional kernel. $R(\mathbf{p})$ has a $3n-m$ dimensional kernel. Therefore, any maximal linearly independent set must have $3n-m-3$ disks. The following should now be obvious:
	
	\begin{proposition}\label{maximalind}
	Given a packing $(G,\mathbf{p})$ with an infinitesimally rigid bar framework, if $S=V^=$ is a maximal set and the complement $S^C=V^0$, then the packing is infinitesimally rigid. 
	\end{proposition}
	\begin{proof}
	Since the bar framework is infinitesimally rigid, the kernel of $\begin{bmatrix}R(\mathbf{p})\\ E_V\end{bmatrix}$ has dimension 3. Since $S$ is maximal, no radii can be added to $S$ to further reduce the dimension of its kernel, so the kernel of $\begin{bmatrix}R(\mathbf{p})\\ E_S\end{bmatrix}$ also has dimension 3. The three dimensions in the kernel can only be trivial infinitesimal motions. 
	\end{proof}
	
	The maximal linearly independent sets of a packing $(G,\mathbf{p})$ have an obvious structure which guarantees that many optimization problems can be solved efficiently over these sets: 
	
	\begin{proposition}
	Given a packing $(G,\mathbf{p})$, linearly independent sets form a matroid. 
	\end{proposition}
	\begin{proof}
	There are many equivalent ways to define a matroid. Here, we use the independent sets to define it. The empty set is independent. The subset of an independent set is clearly independent. Suppose $A$ and $B$ are independent sets such that $|A|>|B|$, then $\begin{bmatrix}R(\mathbf{p})\\ E_A\end{bmatrix}$ has a row $e_k$ not in the row space of $\begin{bmatrix}R(\mathbf{p})\\ E_B\end{bmatrix}$. Then $B\cup \{k\}$ is independent. 
	\end{proof}
	
	Given a packing $(G,\mathbf{p})$ with infinitesimally rigid bar framework and a non-negative real cost function $f:V\to\mathbb{R}^+$ of rigidifying the radius of each disk, then the minimum cost to infinitesimally rigidify the packing can be computed quickly. It is well known that an independent set with minimum cost can always be computed through greedy algorithm if and only if the independent sets form a matroid\cite{matroidGreedy}. Many types of optimization problems can be solved quickly on a matroid\cite{submodularmatroid}. 
\section{Generic Rigidity}\label{sectionGeneric}
    The ideas of this section come primarily from the paper by Connelly and Gortler\cite{sticky_disc}. 
	In \cite{sticky_disc}, it was proved that if a packing of $G$ has the properties $V=V^=$, $m=2n-3$, and all radii are generic (algebraically independent), then it is infinitesimally rigid. This result can be generalized to packings with more than $2n-3$ edges and with fewer fixed radii. The following semi-algebraic version of Sard's Theorem is proved in \cite{sticky_disc}:
	
	\begin{thm}\label{sard}
	Let $X$, $Y$ be smooth semi-algebraic manifolds of dimensions $d_1$ and $d_2$ over $\mathbb{Q}$ and $f:X\to Y$ a rational map. Then the critical values of $f$ are a semi-algebraic subset of $Y$, defined over $\mathbb{Q}$, and of strictly lower dimension than $d_2$.
	\end{thm}
	
	The goal is to prove the following theorem:
	
	\begin{thm}
	Let $(G=(V,E),\mathbf{p})$ be a planar packing with $n>1$ vertices and $m\ge 2n-3$ edges. If there exists $S\subset V$ such that $|S|=3n-m-3$ and the radii of the disks in $S$ are generic, then the packing is infinitesimally rigid by fixing the radii in $S$ and freeing the radii not in $S$. 
	\end{thm}
	
	\begin{proof}
	Let $X_G$ be the set of packings of some fixed orientation defined by the relations $$(x_i-x_j)^2+(y_i-y_j)^2=(r_i+r_j)^2\hspace{1cm}\forall e_{ij}\in E$$  $$r_i>0\hspace{1cm}\forall i\in V$$ It is a smooth manifold of dimension $3n-m$ according to \cite{bauer_stephenson_wegert_2011}. This is clearly a semi-algebraic manifold. Consider the following projection:
	\begin{equation*}
	    f: X_G\to \mathbb{R}^{|S|}
	\end{equation*}
	which maps a packing $\mathbf{p}$ to the radii in $S$. $f$ is a polynomial map. Theorem \ref{sard} states that the critical values of $f$ cannot be generic. Therefore, a generic point must be regular. In particular, $f(\mathbf{p})$ is a regular value. Consider the tangent map:
	
	\begin{equation*}
	    Df: TX_G\to T\mathbb{R}^{|S|}
	\end{equation*}
	
	Since $f(\mathbf{p})$ is a regular value, $\mathbf{p}$ is a regular point. As a result, $Df$ is surjective at $\mathbf{p}$. The dimension of the tangent space of $X_G$ at $\mathbf{p}$, $TX_G$, is $3n-m$ and the dimension of $T\mathbb{R}^{|S|}$ is $|S|=3n-m-3$. Therefore, the kernel has dimension $3n-m-(3n-m-3)=3$. There are always three dimensions of infinitesimal motions at $\mathbf{p}$, which can only be generated by trivial infinitesimal motions. Therefore, the packing $(G,\mathbf{p})$ is infinitesimally rigid. 
	\end{proof}
	
	This result is not too surprising. Since being generic is stronger than linear independence, $3n-m-3$ generic radii are sufficient for infinitesimal rigidity, just as the same number of linearly independent radii is sufficient in proposition \ref{maximalind}. 
	
	Take the set of red and blue disks in Figure \ref{10diskprestress} as an example. If the radii in $V^+\cup V^-$ are generic in another packing of $G$, fixing these radii would be infinitesimally rigid.  The nontrivial infinitesimal motion is caused by linear dependency in a special packing. In this case, it is the reflectional symmetry along the line joining $\mathbf{p}_4$ and $\mathbf{p}_7$. In the next section we will present a method to deal with a special packing that generically should be rigid but is infinitesimally flexible due to being a special arrangement. 
\section{Prestress Stability and Second-order Rigidity}\label{section2ndOrder}
	Suppose a packing $(G,\mathbf{p})$ has a nontrivial infinitesimal flex, it can still be rigid. One such example is given in Figure \ref{10diskprestress} discussed in the case of global rigidity. Our approach for proving the rigidity of Figure \ref{10diskprestress} is not practical for larger packings because finding explicit relationships among radii becomes very difficult as the number of disks grows. This section introduces the idea of second-order rigidity and prestress stability. Taking the second derivative of the tangency constraint, we have:
	\begin{equation*}
	    (\mathbf{p}_i-\mathbf{p}_j)(\mathbf{p}''_i-\mathbf{p}''_j)+(\mathbf{p}'_i-\mathbf{p}'_j)^T(\mathbf{p}'_i-\mathbf{p}'_j)-(r_i+r_j)(r''_i+r''_j)-(r'_i+r'_j)(r'_i+r'_j)=0
	\end{equation*}
	Rearranging the terms gives the following equality:
	\begin{equation}\label{secondorder}
	(\mathbf{p}_i-\mathbf{p}_j)(\mathbf{p}''_i-\mathbf{p}''_j)-(r_i+r_j)(r''_i+r''_j)=(r'_i+r'_j)^2-(\mathbf{p}'_i-\mathbf{p}'_j)^T(\mathbf{p}'_i-\mathbf{p}'_j)
	\end{equation}
	
	Intuitively, the packing would be second-order rigid if no nontrivial proper infinitesimal motion can be extended to a \textquotedblleft proper" second-order motion. To define \textquotedblleft proper" for $\mathbf{p}''$, the constraints on the radii in $\mathbf{p}''$ need to be modified based on a given proper infinitesimal flex $\mathbf{p}'$. If disk $i$ is in $V^+$ and $r'_i>0$, then $r''_i$ should not be constrained because locally $r_i$ increases due to $r'_i>0$ even if $r''_i<0$. Therefore, we need to move disks that strictly change their radii infinitesimally from $V^+$ and $V^-$ to $V^0$. Let us form a new partition based on a given proper $\textbf{p}'$ in the following way: 
    
    $$\tilde{V}^+:=\{i\mid i\in V^+\land r_i'=0 \}$$
    $$\tilde{V}^-:=\{i\mid i\in V^-\land r_i'=0 \}$$
    $$\tilde{V}^=:=V^=$$
    $$\tilde{V}^0:=V^0\cup\{i\mid i\in V^+\land r_i'>0\}\cup\{i\mid i\in V^-\land r_i'<0\}$$
    
    The new partition now is $V=\tilde{V}^+\sqcup \tilde{V}^-\sqcup \tilde{V}^= \sqcup \tilde{V}^0$. Let $R'_e(\mathbf{p})$ be the extended rigidity matrix with the new partition, $\begin{bmatrix}R(\mathbf{p})\\ E_{\tilde{V}^+\cup \tilde{V}^-\cup  \tilde{V}^=}\end{bmatrix}$. 
	\begin{defn}
	Given a packing $(G,\mathbf{p})$ and $\mathbf{p}'$ is a proper infinitesimal motion, $(\mathbf{p}',\mathbf{p}'')$ is \textbf{proper} if (\ref{secondorder}) holds and $r''_i\ge 0 $ for all $ i\in \tilde{V}^+$, $r''_i\le 0 $ for all $ i\in \tilde{V}^-$, $r''_i= 0 $ for all $ i\in \tilde{V}^=$. $\mathbf{p}'$ is \textbf{extendable} if there exists a $\mathbf{p}''$ such that $(\mathbf{p}',\mathbf{p}'')$ is proper. $(G,\mathbf{p})$ is \textbf{second-order rigid} if all non-trivial infinitesimal motions $\mathbf{p}'$ are not extendable. 
	\end{defn}
	
	We should be careful that the partition of the vertices now depends on $\mathbf{p}'$. Now we are ready to prove the following proposition.
	
	\begin{proposition}\label{prestressprop}
	Given a packing $(G,\mathbf{p})$ and $\mathbf{p}'$ is a proper infinitesimal motion, $\mathbf{p}'$ is not extendable if there exists an equilibrium stress $\omega$ such that the radial force sum defined in (\ref{radial}) is non-negative on $\tilde{V}^-$, non-positive on $\tilde{V}^+$, 0 on $\tilde{V}^0$, and satisfies the following inequality:
	\begin{equation}\label{prestress}
	    \sum_{e_{ij}\in E}\omega_{ij}[(\mathbf{p}'_i-\mathbf{p}'_j)^T(\mathbf{p}'_i-\mathbf{p}'_j)-(r'_i+r'_j)^2]>0
	\end{equation}
	\end{proposition}
	\begin{proof}
	First extend the stress as $\omega_e=(...\omega_{ij}...,\omega_k)$ where $\omega_{ij}$ is the stress on $e_{ij}$ indexed over $E$ and $\omega_k$ is the radial force sum defined in (\ref{radial}) for disk $k$ indexed over $\tilde{V}^+\cup \tilde{V}^-\cup  \tilde{V}^=$. $\omega_e$ is in the cokernel of $R'_e(\mathbf{p})$ be definition. If a proper $\mathbf{p}'$ is extendable, then the following equation must hold: \begin{flalign}
	    0&=\omega_e R'_e(\mathbf{p}) \mathbf{p}'' \nonumber\\
	    &=\sum_{e_{ij}\in E}\omega_{ij}[(\mathbf{p}_i-\mathbf{p}_j)(\mathbf{p}''_i-\mathbf{p}''_j)-(r_i+r_j)(r''_i+r''_j)]+\sum_{k\in \tilde{V}^+\cup \tilde{V}^-\cup \tilde{V}^=}\omega_k r''_k
	\label{3product2nd}
	\end{flalign}
	
    \begin{figure}
	        \centering
	        \includegraphics[height=7.5cm]{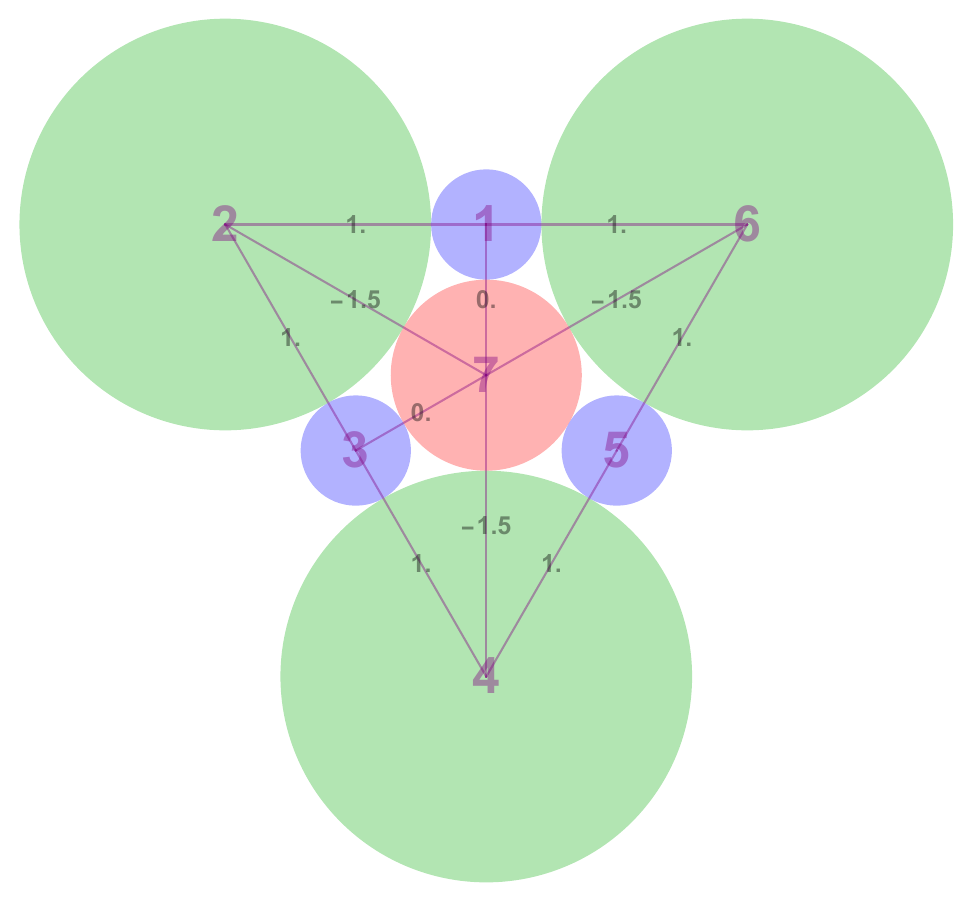}
            \captionsetup{labelsep=colon,margin=1.3cm}
	        \caption{The non-trivial infinitesimal motion is generated by moving disk 5 in the direction away from disk 7. Because disk 5 is in $V^-$, it cannot increase its radius. As a result, this packing is rigid. It can be shown this packing is prestress stable through method introduced in this section.}
	        \label{prestressexample}
	    \end{figure}
	Substituting (\ref{secondorder}) into (\ref{3product2nd}) we have: \begin{equation}\label{substituted3product}
	    0=\sum_{e_{ij}\in E}-\omega_{ij}[(\mathbf{p}'_i-\mathbf{p}'_j)^T(\mathbf{p}'_i-\mathbf{p}'_j)-(r'_i+r'_j)^2]+\sum_{k\in \tilde{V}^+\cup \tilde{V}^-\cup \tilde{V}^=}\omega_k r''_k
	\end{equation}
	The first term is strictly negative by (\ref{prestress}), and the second term is non-positive. Since we reached a contradiction if $\mathbf{p}''$ exists, $\mathbf{p}'$ is not extendable. 
	\end{proof}
	
	    \begin{figure}
	        \centering
	        \includegraphics[height=7.5cm]{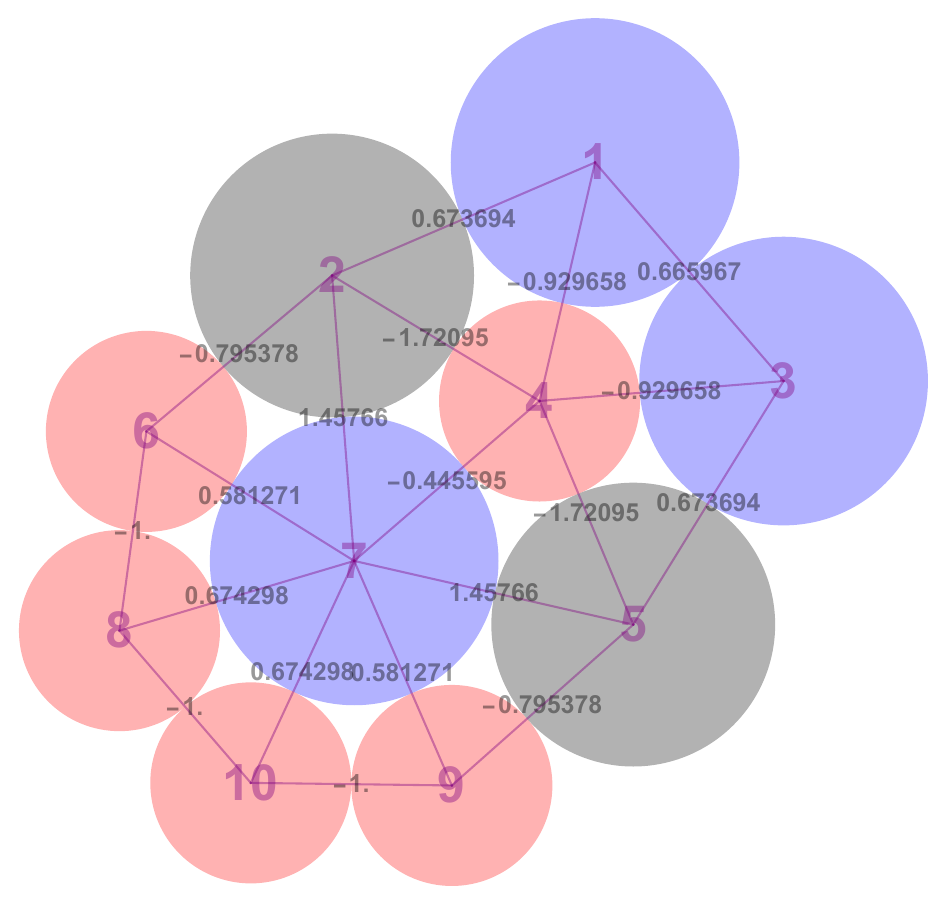}
            \captionsetup{labelsep=colon,margin=1.3cm}
	        \caption{The non-trivial infinitesimal motion is generated by expanding disk 2 while shrinking disk 5 at the same rate. All other radii are fixed infinitesimally. Other disks on the boundary move towards disk 5. However, the stress given shows it is prestress stable.}
	        \label{prestress75}
	    \end{figure}
	    
	\begin{defn}
	If a stress $\omega$ described in Proposition \ref{prestressprop} exists, we say $\omega$ \textbf{blocks} $\mathbf{p}'$. $(G,\mathbf{p})$ is \textbf{prestress stable} if some fixed stress blocks every proper non-trivial infinitesimal motion. 
	\end{defn}
        Here, we emphasize again that the sign requirements on the stress $\omega$ from Proposition \ref{prestressprop} depend on the new partition based on the given $\mathbf{p}'$. Nearly all calculations for second order rigidity are dependent on the $\mathbf{p'}$ being chosen. Therefore, the requirement that the same stress $\omega$ blocking every nontrivial proper $\mathbf{p}'$ is not as simple as it looks when there is more than one nontrivial proper infinitesimal motion. This usually requires a semi-definite programming solver. The reason we care about second-order rigidity is the following:
	
	\begin{proposition}
	If a packing $(G,\mathbf{p})$ is second-order rigid, then it is rigid. 
	\end{proposition}
	\begin{proof} Let $(G,\mathbf{p}(t))$ be a proper motion. Suppose that the lowest-order nonzero derivative is of order $n$.  Without loss of generality, we assume that the first $n-1$ derivatives are all $0$s. In the proof of Proposition \ref{firstorderrig} we showed that the $n^{th}$ derivative must be an infinitesimal motion. The goal is to show that the second-order motion $(\mathbf{p'},\mathbf{p}'')$ will appear at the $n$ and $2n^{th}$ derivative. And therefore a nontrivial proper second order motion exists if a nontrivial proper motion exists.
	$$0= \left( \frac{\partial}{\partial t}\right)^{2n}[(\mathbf{p}_i(t)-\mathbf{p}_j(t))\cdot (\mathbf{p}_i(t)-\mathbf{p}_j(t))-(r_i(t)+r_j(t))(r_i(t)+r_j(t))]$$
		
		\begin{flalign}=\sum_{i=0}^{2n} \binom{2n}{i}\left(\frac{\partial}{\partial t}\right)^i(\mathbf{p}_i(t)-\mathbf{p}_j(t))\cdot \left(\frac{\partial}{\partial t}\right)^{2n-i}(\mathbf{p}_i(t)-\mathbf{p}_j(t))\nonumber\\-\left(\frac{\partial}{\partial t}\right)^i(r_i(t)+r_j(t))\left(\frac{\partial}{\partial t}\right)^{2n-i}(r_i(t)+r_j(t))\end{flalign}
		
		Now if the first $n-1$ degree derivatives are all $0$, then we are left with:
		\begin{gather}
		2\Biggl[(\mathbf{p}_i(t)-\mathbf{p}_j(t))\cdot \left(\frac{\partial}{\partial t}\right)^{2n}(\mathbf{p}_i(t)-\mathbf{p}_j(t))\nonumber\\-(r_i(t)+r_j(t))\left(\frac{\partial}{\partial t}\right)^{2n}(r_i(t)+r_j(t))\Biggr] \nonumber\\=-\binom{2n}{n}\Biggl[\left(\left(\frac{\partial}{\partial t}\right)^n(\mathbf{p}_i(t)-\mathbf{p}_j(t))\right)\cdot\left(\left(\frac{\partial}{\partial t}\right)^n(\mathbf{p}_i(t)-\mathbf{p}_j(t))\right)\nonumber\\
		+\left(\left(\frac{\partial}{\partial t}\right)^n(r_i(t)+r_j(t))\right)^2\,\Biggr]\nonumber
		\end{gather}
		
		Observe that these are the same relations as in equation (\ref{secondorder}) only off by a constant factor. Therefore, $2/\binom{2n}{n}\mathbf{p}^{2n}$ will extend $\mathbf{p}^n$ as a proper second-order motion if $\mathbf{p}^n$ is used as an infinitesimal motion. Since $(G,\mathbf{p})$ is second-order rigid, no such nontrivial infinitesimal flex can exist on order $n$. As a result, $\mathbf{p}(t)$ must be a trivial motion. 
	\end{proof}
	
	Figure \ref{prestressexample} to Figure \ref{prestress777555} are some examples of packings that have a one-dimensional nontrivial infinitesimal motion $\mathbf{p}'$, and the stress blocking $\mathbf{p}'$ is labeled on the edges. Figure \ref{prestressexample} is a case of a small packing that shows the idea behind prestress stability. Disk 5 is moving away from disk 7 in the only non-trivial infinitesimal motion, but disk 4 and disk 6 are \textquotedblleft pulling" it back as soon as disk 5 moves infinitesimally. Disk 5 cannot resist the force from disks 4 and 6 because it cannot increase its radii. Figures \ref{prestress75} and \ref{prestress777555} are much more complex. 
	
	    \begin{figure}
	        \centering
	        \includegraphics[height=7.5cm]{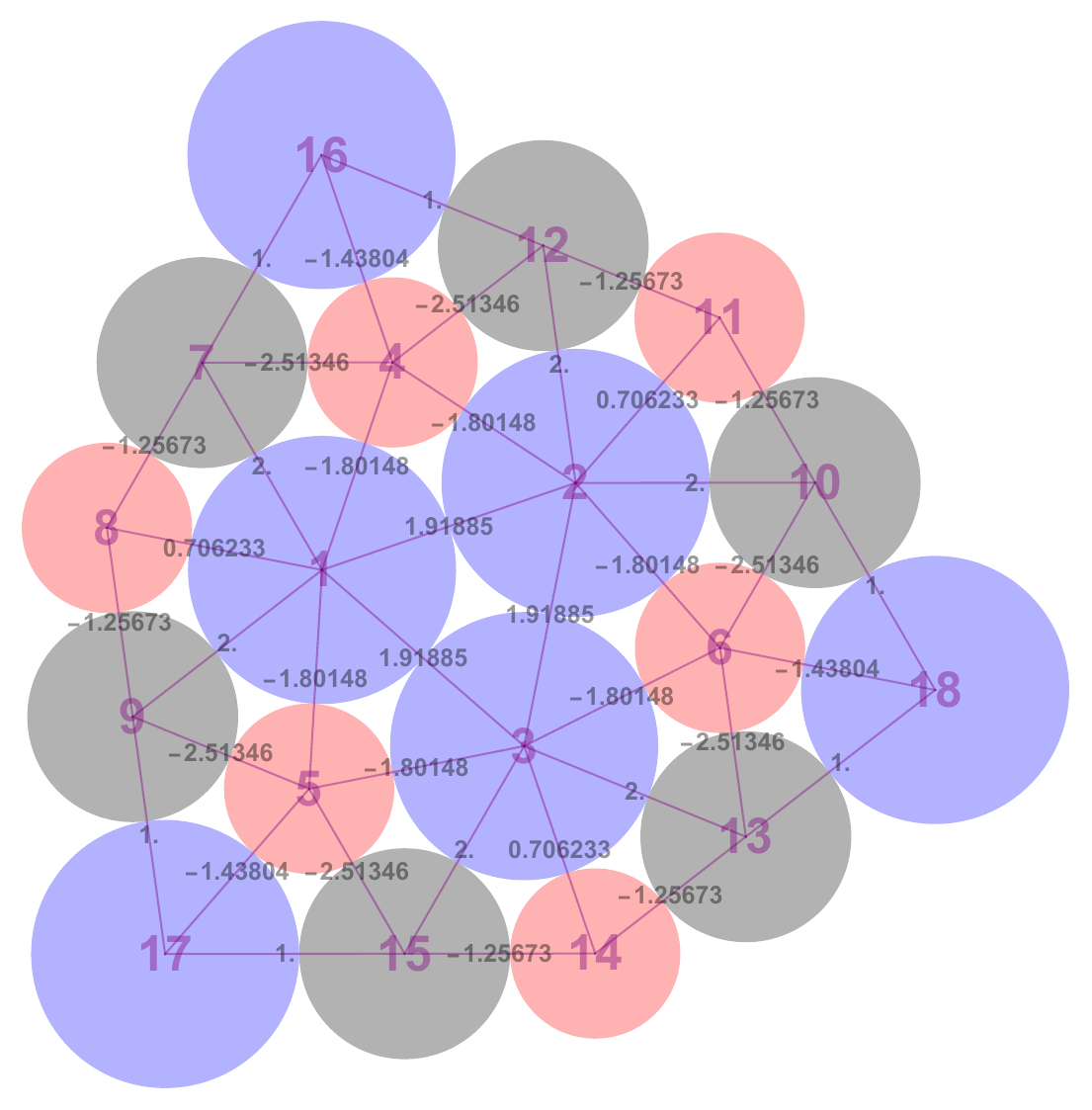}
            \captionsetup{labelsep=colon,margin=1.3cm}
	        \caption{The non-trivial infinitesimal motion is generated by expanding disk 7,10,15 while shrinking disk 9,12,13 at the same rate. All other radii are fixed infinitesimally. The rotational symmetry is preserved infinitesimally. }
	        \label{prestress777555}
	    \end{figure}
	
	An observation is that the switch of $V^+$ and $V^-$ does not preserve the prestress stability. In Figure \ref{prestressexample}, if disk 5 can increase its radius, then pulling disk 5 away from disk 7 while increasing its radius is a valid motion. Suppose that we switch the elements in $V^+$ and $V^-$, then the infinitesimal motion would be $-\mathbf{p}'$ with stress $-\omega$. This causes the first term of Equation \ref{substituted3product} to be positive. Therefore, the stress would no longer block the infinitesimal motion. 
	
	At the end of this section, we prove the converse of Proposition \ref{prestressprop} to establish the duality result for second-order rigidity: 
	\begin{thm}\label{secondOrderDuality}
	Given a packing $(G,\mathbf{p})$ and $\mathbf{p}'$ a non-trivial proper infinitesimal motion, exactly one of the following statements is true:
	 \begin{enumerate}
	    \item $\mathbf{p}'$ is extendable
	    \item	There exists an equilibrium stress $\omega$ blocking $\mathbf{p}'$. 
	\end{enumerate}
	\end{thm}
	\begin{proof}
	Proposition \ref{prestressprop} proved one direction. We prove the other here. First, notice that the statement is similar to theorem \ref{firstOrderDuality}. Therefore, an argument based on Farkas' Alternative should be expected. We use the following version of Farkas' alternative with mixed equalities and inequalities from \cite{secondORigidity}: 
	
	\begin{lem}
	Let $A_0$, $A_+$ be two matrices with $n$ columns and $m_0, m_1$ rows, respectively. $b_0\in\mathbb{R}^{m_0}$, $b_+\in \mathbb{R}^{m_1}$. Then there is a vector $x\in \mathbb{R}^n$ such that
	    $$A_0x=b_0 $$
	    $$A_+x\le b_+ $$
	    
	    if and only if for all vectors $y_0\in \mathbb{R}^{m_0}$, $y_+\in \mathbb{R}^{m_1}$ such that 
	    $$y_0A_0+y_+A_+=0$$
	    $$y_+\ge 0$$
	    
	    we have $y_0b_0+y_+b_+\ge 0$.
	\end{lem}
	Let $A_0=R(\mathbf{p})$, $b_0=\{...,-(\mathbf{p}'_i-\mathbf{p}'_j)^T(\mathbf{p}'_i-\mathbf{p}'_j)+(r'_i+r'_j)^2,... \}$, $A_+=\begin{bmatrix}
		-E_{\tilde{V}^+}\\
		E_{\tilde{V}^-}\\
		E_{\tilde{V}^=}\\
		-E_{\tilde{V}^=}
		\end{bmatrix}$, $b_+=0$, $x=\mathbf{p}''$, $y_0=\omega$. Then the lemma states $\mathbf{p}'$ extends to $\mathbf{p}''$ if and only if for all equilibrium stress $\omega$ satisfying the constraints on the radial force sums must have 
		\begin{equation}
		    \sum_{e_{ij}\in E}\omega_{ij}[-(\mathbf{p}'_i-\mathbf{p}'_j)^T(\mathbf{p}'_i-\mathbf{p}'_j)+(r'_i+r'_j)^2]\ge 0
		\end{equation}
		which can be re-written as 
		\begin{equation}\label{eq15}
		\sum_{e_{ij}\in E}\omega_{ij}[(\mathbf{p}'_i-\mathbf{p}'_j)^T(\mathbf{p}'_i-\mathbf{p}'_j)-(r'_i+r'_j)^2]\le 0
		\end{equation}
		
		Therefore, if $\mathbf{p}'$ is not extendable, some $\omega$ satisfying the radial force sum condition must violate equation (\ref{eq15}). 
	\end{proof}
	
	This establishes the duality for second-order rigidity. The following corollary is now obvious:  
	\begin{corollary}
	A packing $(G,\mathbf{p})$ is second-order rigid if and only if every nontrivial infinitesimal flex $\mathbf{p}'$ is blocked by some stress $\omega$. 
	\end{corollary}
    In the next section, we give an outline of how conjecture \ref{uniformconjecture} is related to global rigidity.
    \section{Global Rigidity and Conjecture 
    \ref{uniformconjecture}}\label{sectionGlobal}
    Although global rigidity is NP-complete for general packings \cite{BreuKirk}, it might be easier for simple triangulated packings. Nearly all known results of global rigidity come from the following observation of simple triangulated packings by Thurston \cite{thurston}:
    \begin{obsv}\label{boundarymonotone}
        For simple triangulated packings, the interior radii are uniquely and strictly monotonically determined by the boundary radii. 
    \end{obsv}
    The uniqueness means no non-congruent packings exist with identical boundary radii. The strict monotonicity means if at least one boundary radius is strictly increased while other boundary radii are not decreased, then all interior radii must strictly increase. In particular, if all boundary disks are in $V^+$ and one interior disk is in $V^-$ (or vice versa), then the packing is globally rigid. 

    Now we state some results related to Conjecture \ref{uniformconjecture}. We assume the largest disks have unit radii. 
    \begin{proposition}\label{vtxdegee}
        Let $G$ be a triangulation of $\mathbb{R}^2$ and $(G,\mathbf{p})$ be a packing with all radii in $[0.62,1]$, then every vertex has degree $5,6$ or $7$. 
    \end{proposition}
    \begin{proof}
        Consider the regular 4-flower, Figure \ref{4flowers}. It is globally rigid since the central disk cannot be larger if none of the boundary disks can increase their sizes according to \ref{boundarymonotone}. If a vertex $v_k$ has degree $4$, then Figure \ref{4flowers} gives the optimal radii ratio that is $\sqrt{2}-1\approx0.414<0.62$. If a vertex has degree $3$, then the central vertex would be even smaller (shrink a boundary radius to 0 and apply the monotonicity of \ref{boundarymonotone}). Similarly, if a vertex has degree at least $8$, then its radius is upper bounded by the case of a regular 8-flower with value $\frac{1}{1-\sin \left(\frac{\pi }{8}\right)}-1<0.62$.
    \end{proof}
    
    Therefore, we can enumerate the finite triangulated subgraphs of a plane triangulation that has a packing with a large radii ratio. If every vertex has degree 6, then the triangulation gives hexagonal packing. Since periodic triangulated packings are unique for a given triangulation\cite{Oded}, degree 5 and 7 vertices must exist if the packing is not the hexagonal packing. Our next observation is that the neighborhood of a degree 5 vertex is very limited. 
    \begin{proposition}\label{two5s}
        Let $G$ be a triangulation of $\mathbb{R}^2$ and $(G,\mathbf{p})$ be a packing with all radii in $[0.64,1]$, then two degree 5 vertices cannot share an edge. 
    \end{proposition}
    \begin{proof}
        The small disks in Figure \ref{fejestothhexagon} have radii approximately $0.637$. By the monotonicity of observation \ref{boundarymonotone}, the two small disks cannot become larger without any boundary disks being bigger than a unit disk. 
    \end{proof}
    \begin{proposition}
        Let $G$ be a triangulation of $\mathbb{R}^2$ and $(G,\mathbf{p})$ be a packing with all radii in $[0.64,1]$, then every degree 5 vertex has at least one degree 7 neighbor. 
    \end{proposition}

    \begin{proof}
        By Propositions \ref{vtxdegee} and \ref{two5s}, a degree 5 vertex satisfying our constraints can only have neighbors of degree 6 or 7. If all the neighbors have degree 6, then the radii ratio $q=\frac{r_{min}}{r_{max}}<0.63$ by Observation \ref{boundarymonotone}, as shown in Figure \ref{d5d7edge}.
    \end{proof}
    \begin{figure}
	        \centering
	        \includegraphics[height=7.5cm]{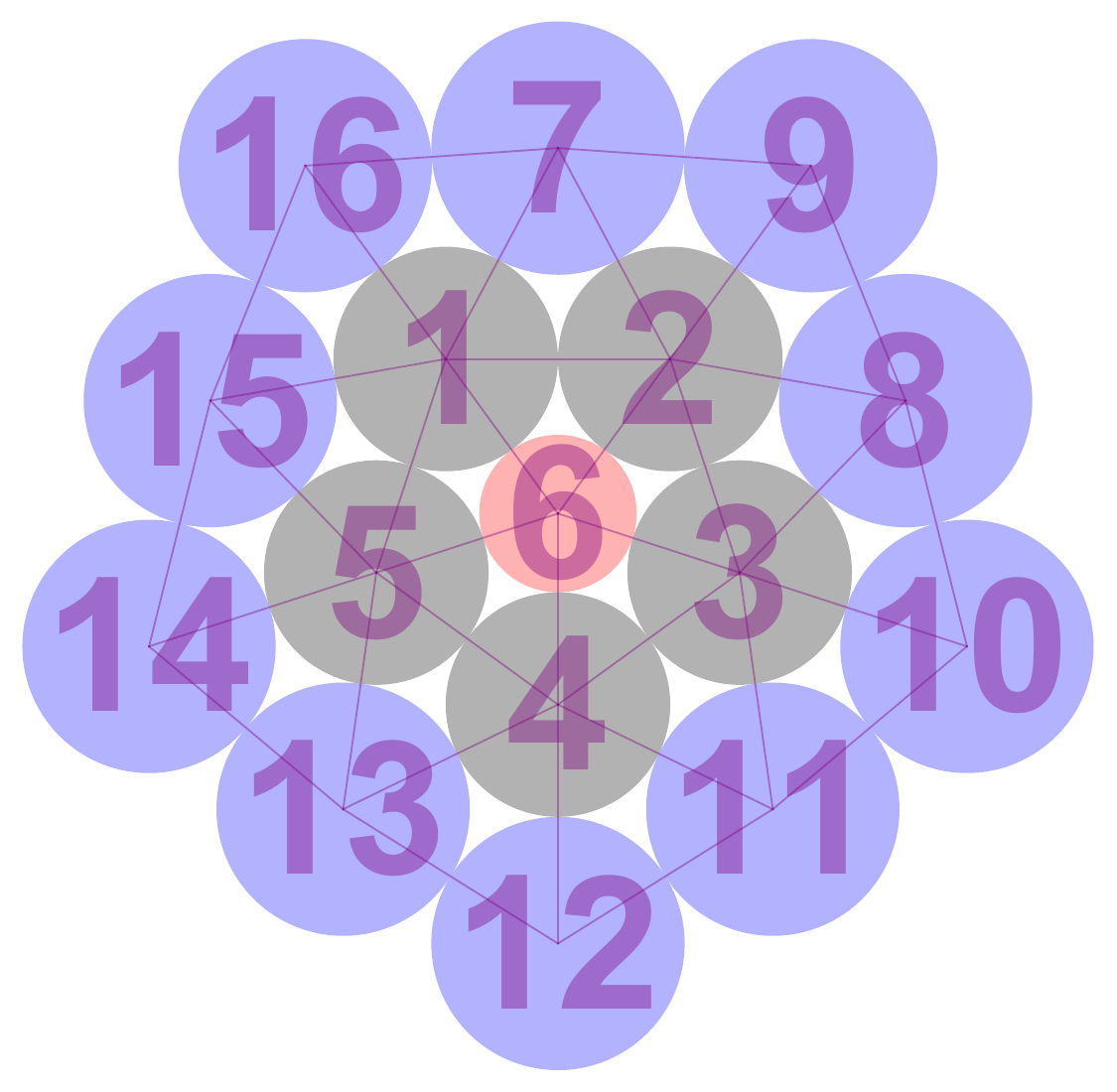}
            \captionsetup{labelsep=colon,margin=1.3cm}
	        \caption{This packing gives the radii ratio $q=r_{min}/r_{max}<0.63$ closest to the one when a degree 5 vertex only has degree 6 neighbors. A larger $q$ requires shrinking the sizes of all boundary disks while not shrinking the central small disk. This is impossible due to Observation \ref{boundarymonotone}.}
	        \label{d5d7edge}
	    \end{figure}

    \begin{proposition}\label{four7prop}
        Let $G$ be a triangulation of $\mathbb{R}^2$ and $(G,\mathbf{p})$ be a packing with all radii in $[0.645,1]$, then four degree 7 vertices cannot form two triangles.
    \end{proposition}
    
    \begin{proof}
        See Figure \ref{four7fig}.
    \end{proof}
    \begin{figure}
	        \centering
	        \includegraphics[height=7.5cm]{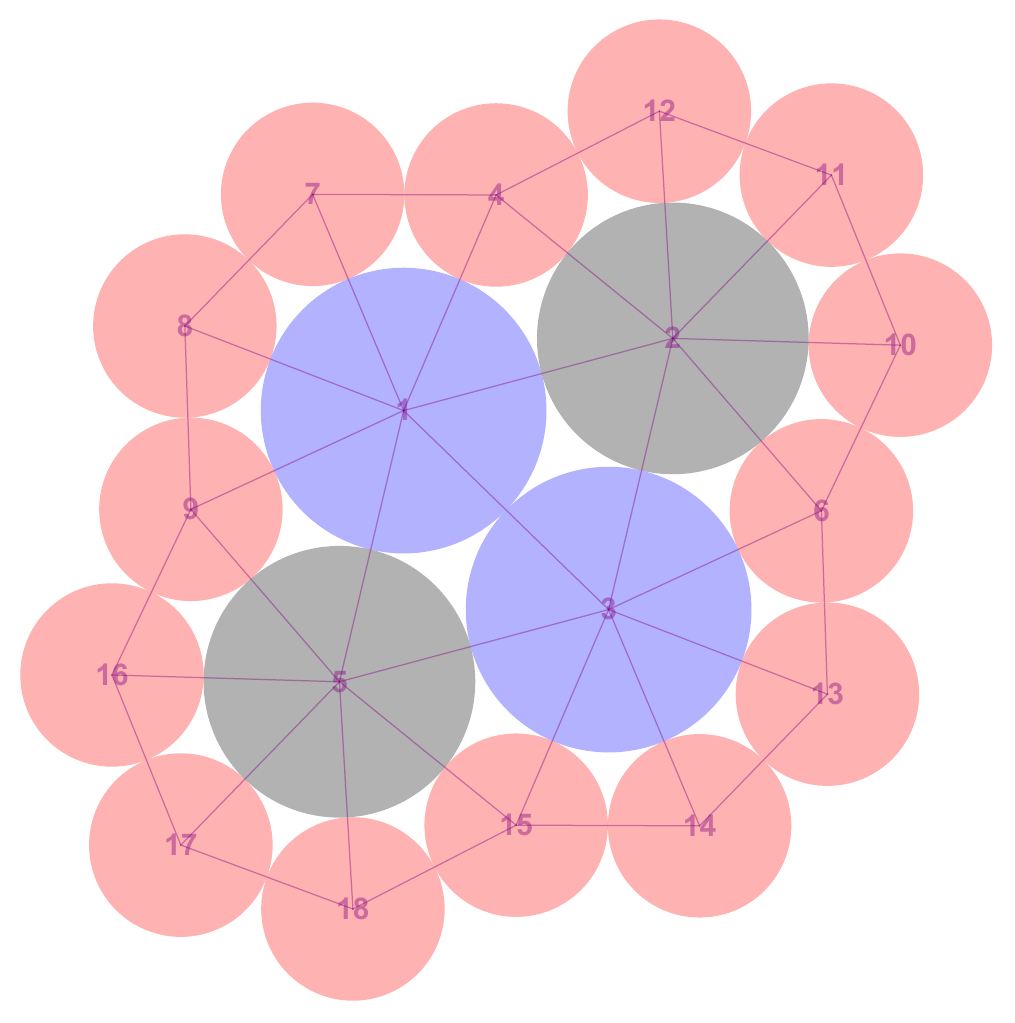}
            \captionsetup{labelsep=colon,margin=1.3cm}
	        \caption{This packing gives the radii ratio $q=r_{min}/r_{max}<0.643$ closest to one when four degree 7 vertices form two triangles. }
	        \label{four7fig}
	    \end{figure}

    Using these results, we can eliminate some finite graphs from possible triangulations that allow a packing to have $q\approx 0.651$. In fact, it is likely that three degree 7 vertices cannot form a triangle. From Figure \ref{three7pic}, it is possible to pack just 3 disks with radii in $[0.689,1]$. However, from Proposition \ref{four7prop}, we know that disks $4,5,6$ can only have degree 5 or 6 if three degree 7 disks do form a triangle in a larger packing. We can further enumerate the degrees of these disks. Figure \ref{prestress777555} is the case where they all have degree 5. In all cases, the optimal packings we can find have $q<0.651$, but we don't have a proof for their global rigidity. 
    
    \begin{figure}[ht]
	        \centering
	        \includegraphics[height=7.5cm]{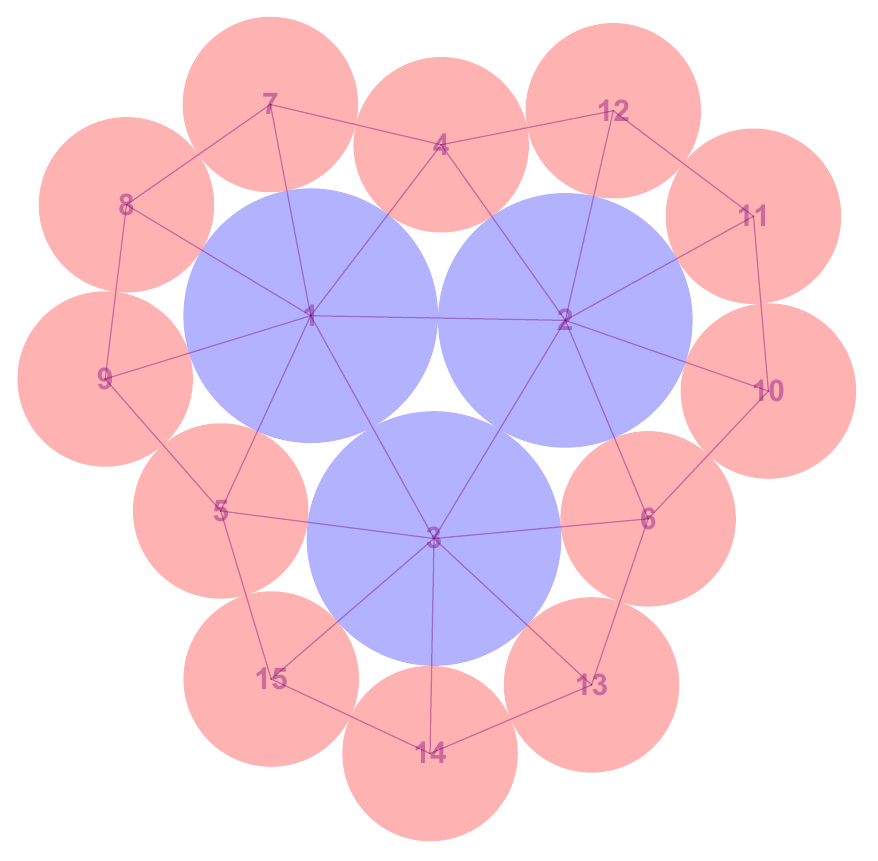}
            \captionsetup{labelsep=colon,margin=1.3cm}
	        \caption{This packing gives the radii ratio $q=r_{min}/r_{max}\approx 0.689$ closest to one when three degree 7 vertices form a triangle.}
	        \label{three7pic}
	    \end{figure}

    If Conjecture $\ref{uniformconjecture}$ is true and we had a way of determining global rigidity, then triangulations different from Fernique's (Figure \ref{fernique}) must be infeasible at some point as we enumerate increasingly larger neighborhoods (Since the degree of each disk is bounded, the number of triangulated graphs for a given number of disks is finite. The sizes of our disks are bounded below, so the set of disks of any triangulation with good radii ratio is countable. If the radii ratio beats Fernique's for infinitely many enumerations, then we arrive at a counterexample). Note that these infeasible triangulations are infeasible everywhere, not just in one place. Therefore, we just need to exclude enough simple triangulated graphs so that the triangulation from a single degree 5 disk can only be extended in one way to produce Fernique's triangulation. We can then assert the uniqueness of Fernique's packing for this triangulation using the result of \cite{Oded}. 
    
	In the next section, we summarize the content of this paper and list some questions based on observations from a fairly large number of finite packings. In particular, if $G$ is a simple triangulated disk, many good properties seem to hold. It is quite likely that some type of ``convexity'' exists in this case that allows many global optimization problems to be solved by a local optimum. 
\section{Summary and Open Problems}\label{sectionDiscussion}
    First, let us summarize the implication relationships between various notions of rigidity:
    \begin{align*}
        &\text{Infinitesimal Rigidity} \implies \text{Prestress Stability} \\&\implies  \text{Second Order Rigidity} \implies \text{Local Rigidity}
    \end{align*}

    $$\text{Global Rigidity} \implies \text{Local Rigidity}$$

    Infinitesimal rigidity trivially implies prestress stability since the set of nontrivial infinitesimal motions is empty, so every nontrivial infinitesimal motion is blocked. Prestress stability implies second-order rigidity by definition: prestress stability requires one stress to block all nontrivial infinitesimal motions, while second-order rigidity requires every nontrivial infinitesimal motion to be blocked by some stress. A globally rigid packing is obviously locally rigid, but it can be locally rigid for various reasons: it can be infinitesimally rigid, but it can also be not even second-order rigid. 

    Next, we give intuitive explanations, from an engineering point of view, of different types of structural local rigidities in general, not limited to circle packings. Readers interested in rigidity theory are encouraged to read \cite{connelly_2022}. In Mathematics, all types of rigidities assert local uniqueness because we assume our constraints are absolute. In real-world structures, materials always deform if external forces act on them. Infinitesimal rigidity is also known as ``static rigidity'' by engineers. Intuitively, statically rigid structures can generate an internal force that counterbalances any external forces (called external equilibrium loads) trying to induce a nontrivial flex, without any deformation themselves. This is what engineers often desire for a stable structure, especially when the material is not particularly stiff. Prestress stable structures often require a ``pre-existing stress'' to help rigidify the structure against a nontrivial infinitesimal motion. For example, one can wrap a disk in $V^+$ with a tight rubber band. This ``pre-existing stress'' counters potential infinitesimal motions. Another way of rigidifying prestress stable structures is using stiff materials that produce large resistant forces under external loads. When the structure moves infinitesimally, a strong resistant force is produced when the deformation is tiny. Second-order rigid structures may not feel rigid at all. 
    
	Our paper proves the duality of motions and stresses in infinitesimal and second-order cases for circle packings in the plane. For a general rigidity problem that is over-constrained, it is likely that the stress problem is as hard as the radius problem. However, in the context of local optimization over some function defined on a packing, it is almost guaranteed that the resulting packing only has one stress unless the graph is highly symmetric. This is because the optimization process stops as soon as the first stress that asserts local optimality exists. It is unlikely that multiple stresses occur exactly at the same time without some symmetry. 
	
	\begin{example}
	
	Maximizing $q=r_{\min}/r_{\max} $ for the corresponding planar embedded graphs yields Figures \ref{4flowers}, \ref{10diskprestress}, and \ref{prestress777555}. 
	\end{example}
	In these examples, there are two cases. For Figure \ref{4flowers}, the result is infinitesimally rigid, and so there are $3n-m-2$ inequality constraints. This is generally the case for most graphs that do not have symmetry. For Figures \ref{10diskprestress} and \ref{prestress777555}, they are infinitesimally flexible with only $3n-m-3$ inequality constraints. However, the only stress produced in the optimization process is able to block the infinitesimal motion. 
	The following question is critical to prove conjecture \ref{uniformconjecture}.
	
	\begin{question}
	If a packing has locally maximal $r_{min}/r_{max}$, is it also globally maximal (for a given orientation)? If not, is this at least true for simple triangulated packings? 
	\end{question}
	
	Note that in the case of simple triangulated packings, the set of packings with boundary radii in $[0,1]$ is convex, but the function $q$ is not concave. So we cannot expect a local maximum point to be a global one. However, our numerical simulations show $q$ ``looks almost concave'' for simple triangulated packings with large $q$. 
    
    The following question is based on observations:
	
	\begin{question}
	Does local rigidity imply infinitesimal or second-order rigidity for simple triangulated packings?
	\end{question}
	
	Generally, packings that are rigid but not rigid to the first or second order can be generated from examples of such bar frameworks (see \cite{higherorder}). However, the bar framework is necessarily rigid infinitesimally for simple triangulated packings. Therefore, counterexamples are much harder to find even if the answer is negative.

	\begin{example}
	Maximizing the sum of radii with $0\le r_i\le 1$ yields the packing in Figure \ref{sumr}. 
	\end{example}
	In Figure \ref{sumr}, there are $27$ vertices and $61$ edges. In this case, $3n-m-3=17$ is exactly the number of disks with radius $1$.  Together with the radius sum constraint, there are $3n-m-2$ constraints. Adding an extra row of $1$ to every radii corresponding to radii sum to $R_e(\mathbf{p})$ gives exactly one stress. This is generally the result of this optimization process. Unlike the radii ratio function, the radius sum function does not seem to preserve symmetry. As a result, even rigid infinitesimally flexible packings are hard to find. This motivates the following question:
	\begin{question}
	If a simple triangulated packing($m\ge 2n-3$) with no radius greater than $1$ has a maximal radius sum locally, does it always have at least $3n-m-3$ disks of radius $1$?
	\end{question}

	\begin{figure}
		\centering
		\includegraphics[height=7.5cm]{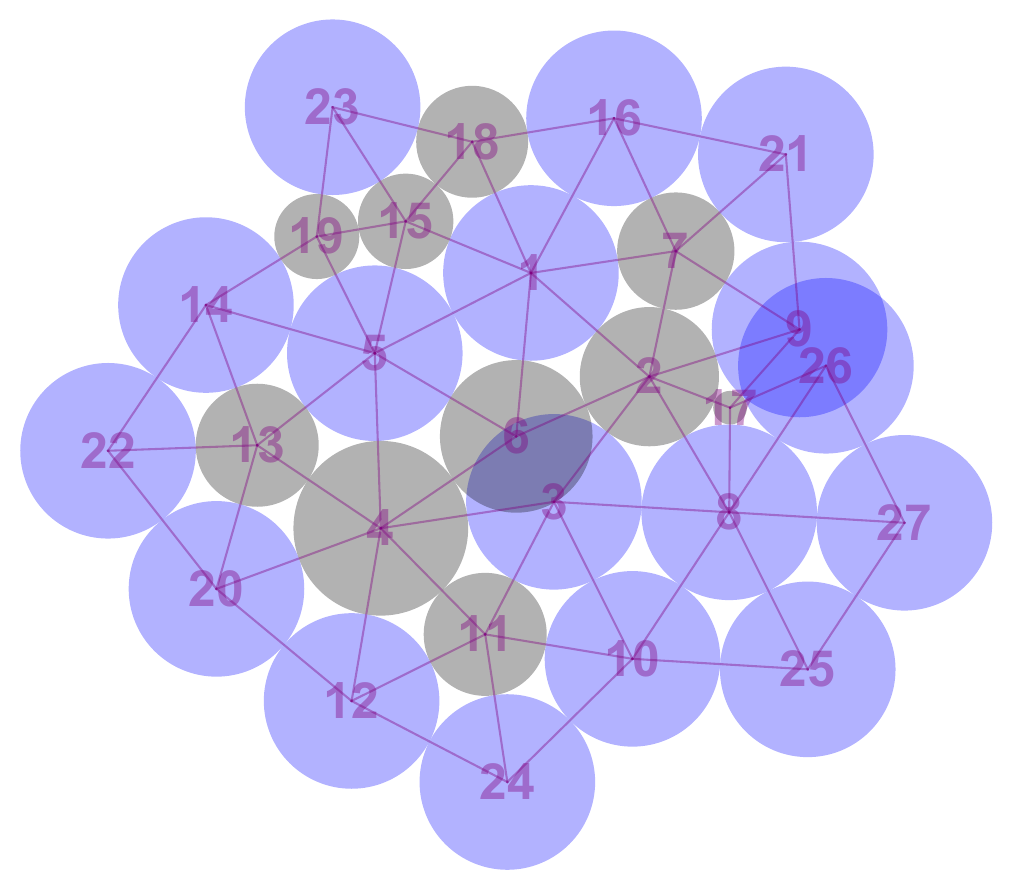}
        \captionsetup{labelsep=colon,margin=1.3cm}
		\caption{The result of local optimization over the sum of radii. There are $17=3n-m-3$ disks with radius 1. Note that our disks are not disjoint, but the orientation is identical to that of \textit{some planar packing} so the rows of the packing rigidity matrix $R(\mathbf{p})$ are still independent, giving a kernel of dimension $3n-m$.}
		\label{sumr}
	\end{figure}

	Next we give an answer to a more general version of the above questions: If a packing is locally rigid, is it globally rigid? The answer is unfortunately false given the following counterexample:

	\begin{example}
	Figure \ref{fixedradiusoptimal} has two packings that are reflections of each other. Both are rigid locally. The packing on the right is not globally rigid because of the packing on the left. 
	\end{example}
	    
		\begin{figure}
	        \centering
	        \includegraphics[height=5cm]{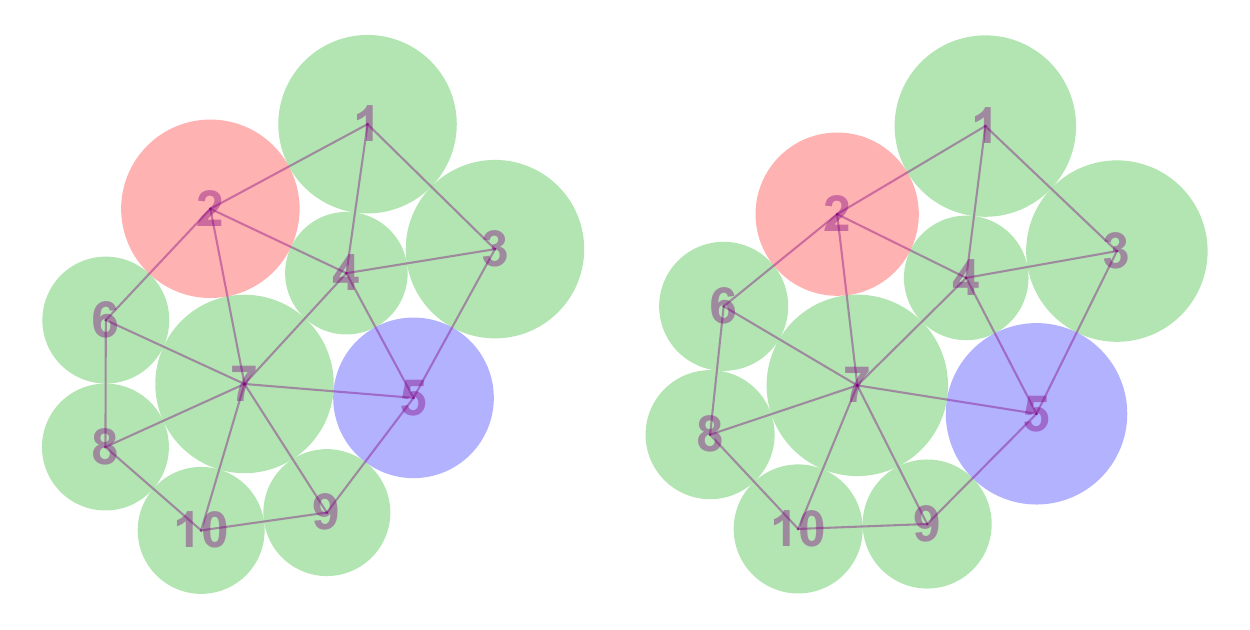}
            \captionsetup{labelsep=colon,margin=1.3cm}
	        \caption{The green disks have identical radii in both packings}
	        \label{fixedradiusoptimal}
	    \end{figure}
	    
    Figure \ref{fixedradiusoptimal} shows that the answer to the question is negative. Green disks with the same label have identical radii in both packings. They form a maximal independent set infinitesimally (unless $r_2=r_5$) and therefore must locally determine the packing. 
	
	Another observation is that global rigidity depends on the specific constraints on the radii. Although both packings in Figure \ref{fixedradiusoptimal} are infinitesimally rigid and mirror reflections of each other, the figure on the left is globally rigid because disk 2 is in $V^+$ (labeled red) so the packing on the right is not valid, while the figure on the right is not globally rigid since the figure on the left is valid for its constraints. 
	
	There are also cases where local rigidity clearly implies global rigidity. If the packing is simple and triangulated, the results from \cite{thurston, bauer_stephenson_wegert_2011,kstep} show that the boundary radii determine the interior radii strictly and monotonically. Therefore, if all boundary radii are bounded above with at least one interior radius bounded below or vice versa, then local rigidity implies global rigidity. One such example is Figure \ref{4flowers}. 
	
	Another example is Figure \ref{10diskprestress}, which was shown to be globally rigid using arguments based on angle sums in section 4. Our observation is that when the symmetries of the graph are preserved in the packing, local rigidity appears to imply global rigidity. 
	
	The next question concerns the existence of generic packings and rigidity when removing an edge. 
	\begin{question}
	    Given a planar embedded generically rigid graph $G$ with $n$ vertices and $m\ge 2n-3$ edges, is there a circle packing of $G$ with $3n-m-3$ generic radii? If so, is the set of such packings dense in the set of packings of $G$?
	\end{question}
	
	Intuitively, the answer to both of these questions should be affirmative if the graph is ``nice" enough. We give a heuristic argument on why the answers might be yes. For simple triangulated packings, the answers are obvious given there are $3n-m-3$ boundary radii and that boundary radii can be arbitrary \cite{bauer_stephenson_wegert_2011,kstep}. For graphs that are not triangulated, the idea is to remove edges from a triangulated one. Based on Lemma \ref{kerneldimension}, we know that removing an edge will always free a radius infinitesimally. The question then becomes whether or not it is always possible to extend this infinitesimal flex to a real flex. 
	
	Using Figure \ref{prestressexample} with an extra edge $e_{57}$ as an example, removing $e_{57}$ still keeps the packing rigid to the second order. However, this seems to require the packing to be not generic before we remove the edge. When disk $5$ is not on the same line as disk $4$ and disk $6$, removing the edge $e_{57}$ creates an extra degree of freedom. Notice that here we need $G$ to be planar embedded in order for Lemma \ref{kerneldimension} to hold. Without this assumption, it is possible to have duplicated disks sharing 3 neighbors that can never have generic radii (see Figure 6 from \cite{sticky_disc}).  

    \section*{Declarations}
    The authors declare that the data supporting the findings of this study are available within the paper. Specific implementations of the methods presented in this paper are available from the corresponding author upon reasonable request. 
	
	\hfill
	\printbibliography
\end{document}